\documentclass[11pt,reqno]{amsart}
\usepackage{amsmath,amsthm,amsfonts,amssymb,mathrsfs,bm,graphicx,stmaryrd}

\usepackage{mathtools}
\usepackage{dsfont}
\usepackage{multicol}
\usepackage[colorlinks=true,linkcolor=blue]{hyperref}
\hypersetup{bookmarksdepth=1}
\usepackage{enumitem}
\usepackage{bbm}
\usepackage{mathrsfs}
\usepackage{versions}
\usepackage{subcaption}

\usepackage[letterpaper,hmargin=1.0in,vmargin=1.0in]{geometry}
\parskip 	\smallskipamount

\renewcommand{\P}{\mathbb{P}}

\newcommand{\cf}{\mathcal{F}}

\newcommand{\ce}{\mathcal{E}}

\newcommand{\N}{\mathbb{N}}
\newcommand{\R}{\mathbb{R}}
\newcommand*{\wih}{\widehat}
\newcommand{\cl}{\mathcal{L}}

\newcommand{\ca}{\mathcal{A}}
\newcommand{\cb}{\mathcal{B}}
\newcommand{\cc}{\mathcal{C}}

\newcommand{\vep}{\varepsilon}

\def\l{\left}
\def\r{\right}
\def\<{\langle}
\def\>{\rangle}

\newtheorem{theorem}{Theorem}[section]

\newtheorem{lemma}[theorem]{Lemma}
\newtheorem{proposition}[theorem]{Proposition}

\newtheorem{observation}[theorem]{Observation}

\theoremstyle{remark}
\newtheorem{remark}{\bf Remark}

\numberwithin{equation}{section}

\title[LFL in random matrices and random permutations]{The Paquette-Zeitouni law of fractional logarithms for the GUE minor process and the Plancherel growth process}
 \author[Baslingker, Basu, Bhattacharjee, Krishnapur]{Jnaneshwar Baslingker, Riddhipratim Basu,\\ Sudeshna Bhattacharjee, Manjunath Krishnapur}

\begin{document}
\begin{abstract}
  It is well-known that the largest eigenvalue of an $n\times n$ GUE matrix and the length of a longest increasing subsequence in a uniform random permutation of length $n$, both converge weakly to the GUE Tracy-Widom distribution as $n\to \infty$. We consider the sequences of the largest eigenvalues of the $n\times n$ principal minor of an infinite GUE matrix, and the the lengths of longest increasing subsequences of a growing sequence of random permutations (which, by the RSK bijection corresponds to the top row of the Young diagrams growing according to the Plancherel growth process), and establish a law of fractional logarithms for these. That is, we show that, under a further scaling of $(\log n)^{2/3}$ and $(\log n)^{1/3}$, the $\limsup$ and $\liminf$ respectively of these scaled quantities converge almost surely to explicit non-zero and finite constants. Our results provide complete solutions to two questions raised by Kalai in 2013. We affirm a conjecture of Paquette and Zeitouni \cite[Conjecture 1.3]{PZ17}, and give a new proof of \cite[Theorem 1.1]{PZ17} who provided a partial solution in the case of GUE minor process. 
  \end{abstract}

 \address{Jnaneshwar Baslingker,  Department of Mathematics, Indian Institute of Science, Bangalore, India}
 \email{jnaneshwarb@iisc.ac.in}
 \address{Riddhipratim Basu, International Centre for Theoretical Sciences, Tata Institute of Fundamental Research, Bangalore, India} 
 \email{rbasu@icts.res.in}
 \address{Sudeshna Bhattacharjee, Department of Mathematics, Indian Institute of Science, Bangalore, India}
 \email{sudeshnab@iisc.ac.in }
 \address{Manjunath Krishnapur, Department of Mathematics, Indian Institute of Science, Bangalore, India}
 \email{manju@iisc.ac.in}
\maketitle

\section{Introduction and Main Results}
Consider a doubly infinite array of random variables $\mathbf{Z}=(Z_{i,j})_{i,j\ge 1}$ where $Z_{i,i}$ are i.i.d. standard (real) Gaussian random variables with mean zero and variance 1, $Z_{i,j}$ for $i>j$ are i.i.d.\ standard complex Gaussian random variables (i.e., the real and imaginary parts are independent centred Gaussian variables with variance $\frac{1}{2}$ each) and $Z_{i,j}=\overline{Z}_{j,i}$ for all $i,j$ ($\bar{z}$ denotes the complex conjugate of $z$). For $n\ge 1$, let $G_{n}$ denote the $n\times n$ sub-matrix given by  $(Z_{i,j})_{1\le i,j\le n}$. Then, $G_{n}$ is a standard $n\times n$ GUE (Gaussian Unitary Ensemble) matrix, and has $n$ real eigenvalues. The collection $\{G_n \}_{n \geq 1}$ is known as the \textit{GUE minor process}. Let $\lambda_n$ denote the largest eigenvalue of $G_{n}$. It is well known (Theorems~3.1.4 and 3.1.5 of \cite{AGZ09}) that 
$$n^{1/6}(\lambda_n-2\sqrt{n})\Rightarrow \mbox{TW}_2$$
in distribution where $\mbox{TW}_2$ is the GUE Tracy-Widom distribution. 

The Tracy-Widom distribution is a universal scaling limit distribution that arises in many other contexts (e.g.\ in the Kardar-Parisi-Zhang universality class). Here is another well known example. Consider a growing sequence of permutations $\sigma_n\in S_n$ (the group of symmetries of $\{1,2,\dots,n\}$), where $\sigma_{n+1}$ is obtained from $\sigma_n$ by inserting the element $(n+1)$ at a randomly chosen location, and shifting all subsequence elements by $1$. It is easy to see that the marginal of $\sigma_n$ is uniform on $S_n$. Let $L_n$ denote the length of a longest increasing subsequence of $\sigma_n$. It is known that \cite{BDJ98}
$$ \frac{L_n-2\sqrt{n}}{n^{1/6}}\Rightarrow  \mbox{TW}_2$$
in distribution.

Kalai \cite{Kalai_question} raised the question of finding an analogue of the  classical law of iterated logarithm in these examples. That is, denoting the scaled quantities above by $\widetilde{\lambda}_n$ and $\widetilde{L}_n$ respectively, do there exist (non-random) scaling functions $f_{+}(\cdot)$ and $f_{-}(\cdot)$ such that 
\begin{equation}
    \label{e:limits}
    \limsup_{n\to \infty}\frac{\widetilde{\lambda}_n}{f_{+}(n)}\quad  \mbox{ and } \quad \liminf_{n\to \infty}\frac{\widetilde{\lambda}_n}{f_{-}(n)}
\end{equation}
almost surely equal non-zero finite constants? What about the quantities 
\begin{equation}
    \label{e:limits1}
    \limsup_{n\to \infty}\frac{\widetilde{L}_n}{f_{+}(n)}\quad  \mbox{ and } \quad \liminf_{n\to \infty}\frac{\widetilde{L}_n}{f_{-}(n)}
\end{equation}
with possibly different scaling functions $f_{\pm}(\cdot)$?
The first result in this direction in the case of the GUE minor process was by Paquette and Zeitouni \cite{PZ17}. They showed that

\begin{equation}
    \label{e:PZ1}
    \limsup_{n\to \infty}\frac{\widetilde{\lambda}_n}{(\log n)^{2/3}}=\left(\frac{1}{4}\right)^{2/3}
\end{equation}
 almost surely, whereas for the lower deviations they showed the weaker result that there are finite positive constants $c_1$ and $c_2$ such that  

\begin{equation}
    \label{e:PZ2}
    \liminf_{n\to \infty}\frac{\widetilde{\lambda}_n}{(\log n)^{1/3}}\in (-c_1,-c_2)
\end{equation}
almost surely. These results were dubbed \emph{law of fractional logarithms} in \cite{PZ17}. They conjectured that the constant exists and equals $-4^{1/3}$ for the liminf \cite[Conjecture 1.3]{PZ17} and our first main result confirms the Paquette-Zeitouni conjecture. 

\begin{theorem}\label{thm: liminf constant}
Almost surely, in the above set-up, 
\begin{align*}
    \liminf_{n\rightarrow\infty}\frac{\left({\lambda_n-2\sqrt{n}}\right){n^{1/6}}}{(\log n)^{1/3}}=-4^{1/3}.
\end{align*}
\end{theorem}

We shall also provide a new proof of \eqref{e:PZ1} (see Theorem \ref{t:LPPupper}), and answer Kalai's question for growing random permutations (see Theorem \ref{t: LFL for LIS}). Since the arguments in the case of random permutations will be similar, we shall continue with describing our results and techniques for the GUE minor process case for now; the case for the longest increasing subsequence is discussed in Section \ref{s:lis}.

We prove Theorem \ref{thm: liminf constant} by showing separately that for arbitrary $\varepsilon>0$, almost surely 
$$\liminf_{n\rightarrow\infty}\frac{\left({\lambda_n-2\sqrt{n}}\right){n^{1/6}}}{(\log n)^{1/3}}>-(4+\varepsilon)^{1/3}$$ and  
$$\liminf_{n\rightarrow\infty}\frac{\left({\lambda_n-2\sqrt{n}}\right){n^{1/6}}}{(\log n)^{1/3}}<-(4-\varepsilon)^{1/3}.$$

We also show a general $0-1$ law (see Proposition \ref{p:01}) which shows that under certain conditions on the scaling functions $f_{\pm}$ in \eqref{e:limits} and \eqref{e:limits1} the limits are almost surely constants, but this will not be used to prove our main results.

Our work follows the general framework of \cite{PZ17}, which, like the classical proof of the LIL for simple random walk, consisted of two steps, \emph{sharp tail estimates} and \emph{decorrelation at the correct scale}. It is known \cite{RRV} that for $\beta>0$, the $\beta$ Tracy-Widom distribution (GUE Tracy-Widom distribution corresponds to $\beta=2$) satisfies $\P(\mbox{TW}_{\beta}>a)\approx \exp\l(-\frac{2\beta a^{3/2}}{3}\r)$ and $\P(\mbox{TW}_{\beta}<-a)\approx \exp\l(-\frac{\beta a^{3}}{24}\r)$ as $a\to \infty$. Here, $a_n\approx b_n$ means that $\log a_n/\log b_n$ converges to $1$ as $n\to \infty$. One might therefore expect that for appropriately large values of $a \ll n$ we have 

\begin{equation}
    \label{e:sharp}
    \P(\widetilde{\lambda}_n>a) \approx \exp\l(-\frac{4a^{3/2}}{3}\r); \quad \P(\widetilde{\lambda}_n<-a) \approx \exp\l(-\frac{a^{3}}{12}\r). 
\end{equation}
These will be referred to as the \emph{sharp tail estimates}. \emph{Decorrelation at the correct scale} refers to finding  the scale $k(n)$ such that the random variables $\widetilde{\lambda}_n$ and $\widetilde{\lambda}_{n+k}$ are strongly correlated if $k\ll k(n)$ (in the sense that $\max \limits_{n\le i \le n+k} \widetilde{\lambda}_i$ and $\min \limits_{n\le i \le n+k} \widetilde{\lambda}_i$ can be well approximated by $\widetilde{\lambda}_n$) and are decorrelated if $k\gg k(n)$.

The sharp estimates in \eqref{e:sharp} for the upper tail was known and \cite{PZ17} used this together with a Fredholm determinant formula for the joint distribution of the largest eigenvalues to show decorrelation of the upper tail events $\ce_n=\{\widetilde{\lambda}_n >a_{n}\}$ (for appropriately chosen values of $a_n$) at scale $k(n)=n^{2/3}$ (in fact they showed that $\widetilde{\lambda}_{n_{k}}$ are strongly correlated if $n_{k}=k^{3-\varepsilon}$ and the events $\ce_n$ are approximately independent if $n_{k}=k^{3+\varepsilon}$, which is essentially equivalent). For the lower tail, however, the sharp tail estimates from \eqref{e:sharp} were not available, and \cite{PZ17} instead used a weaker tail estimate from \cite{LR10} (which had the correct exponent $a^3$ but not the correct constant) together with a suboptimal decorrelation estimate for the lower tail events $\cf_n=\{\widetilde{\lambda}_n <-b_{n}\}$ (for a differently chosen sequence $b_n$) to obtain the weaker conclusion described above. Nevertheless, \cite{PZ17} conjectured that the decorrelation also happens for the lower tail events at the same scale $k(n)=n^{2/3}$ which led to their conjecture for the limiting constant in the liminf case. 

Among the two missing ingredients needed to establish the Paquette-Zeitouni conjecture, the sharp tail estimates were recently established in \cite{BBBK24} (in fact \cite{BBBK24} had optimal lower tail estimates of Gaussian $\beta$ ensembles for all $\beta>0$), and in this paper we provide the decorrelation argument establishing Theorem \ref{thm: liminf constant}. Interestingly, we do not establish the decorrelation estimate using random matrix theory, instead we rely on a well-known distributional identity,  due to Baryshnikov \cite{Bar01},  between the largest eigenvalues of the GUE minor process and a vector of passage times in Brownian last passage percolation.

\medskip
\noindent
\textbf{Brownian last passage percolation:}
Let $\{B_i(t):t\geq 0\}_{i\geq 1}$ be i.i.d.\ standard Brownian motions. 
For $t>0$ and $n\ge 1$, define 
\begin{align*}
    D_{t,n}:=\sup _{0=t_0\le t_1\le \dots \le t_{n-1}\le t_n=t}\sum\limits_{i=1}^{n}\left[B_i(t_{i})-B_i(t_{i-1})\right]
\end{align*}
which is the maximum value one can obtain by summing Brownian increments over a partition of $[0,t]$ into $n$ intervals\footnote{Sometimes the numbers $t_i$ in the definition are strictly ordered, e.g.\ in \cite{Bar01}, but it is easy to see by the continuity of the Brownian motion trajectories, that this is equivalent to the above definition} where the Brownian increment contribution from the $i$-th interval comes from the increment of the Brownian motion $B_i(\cdot)$. Interpreting the collection of intervals $[t_{i-1},t_i] \times \{i \}$ as a directed path in $\R_+\times \N$, one can consider $D_{t,n}$ as the last passage time from $(0,1)$ to $(t,n)$ in the Brownian last passage percolation model with noise $\{B_{i}(\cdot)\}$. One can similarly define last passage time from $(s,k)$ to $(t,n)$ denoted $D_{(s,k)\to (t,n)}$ for $t\ge s$ and $n\ge k$. Brownian LPP is one of the few exactly solvable models in the (1+1)-dimensional KPZ universality class. The connection between Brownian LPP and GUE minor process is given by the following result from \cite{Bar01} (see also \cite{AVMW13} for a more general result).

\begin{proposition}[{\cite[Theorem 0.7]{Bar01}}]
    \label{p:BLPP}
    With the above notation, we have 
    $(\lambda_n)_{n\ge 1}\stackrel{d}{=} (D_{1,n})_{n\ge 1}$. 
\end{proposition}

Using the sharp tail estimates for $D_{1,n}$ from \cite{BBBK24} together with decorrelation estimates obtained using the geometry of geodesics (optimal paths) in Brownian LPP we shall prove the following result, which, by Proposition \ref{p:BLPP} is a restatement of Theorem \ref{thm: liminf constant}.

\begin{theorem}
    \label{t:LPP}
    Almost surely, 
    \begin{align}
\liminf_{n\rightarrow\infty}\frac{\left({D_{1,n}-2\sqrt{n}}\right){n^{1/6}}}{(\log n)^{1/3}}=-4^{1/3}.
\end{align}
\end{theorem}

Observe that, from the point of view of last passage percolation, Theorem \ref{t:LPP} is not a natural result. It would be more natural to consider the extremal behaviour of the sequence $\{n^{-1/3}(D_{n,n}-2n)\}$, which, by the scale invariance of Brownian motion, has the same marginal distribution as $\left({D_{1,n}-2\sqrt{n}}\right){n^{1/6}}$ but has a different coupling.

Indeed, such questions were considered (mostly in the context of exponential LPP); see \cite{L18,BGHK21,BBBK24}. However, in this case the decorrelation scale turns out to be linear i.e., $k(n)\sim n$ instead of a sub-linear polynomial and one therefore gets a \emph{law of iterated logarithm} (instead of a law of fractional logarithm) as in the classical simple random walk case. Indeed, following the earlier works \cite{L18,BGHK21}, a full statement of a law of iterated logarithm in that case was established in \cite{BBBK24} using tail estimates and geometry of geodesics. Although we use the same ingredients, our decorrelation argument is qualitatively different from those in \cite{BBBK24}. Further, using our general decorrelation statement (see Proposition \ref{p:decorr}) we shall also give a different proof of \eqref{e:PZ1}. The following theorem, by Proposition \ref{p:BLPP}, is restatement of \eqref{e:PZ1} (which is also \cite[Theorem 1.1]{PZ17}).  

\begin{theorem}
    \label{t:LPPupper}
    Almost surely, 
    \begin{align}
\limsup_{n\rightarrow\infty}\frac{\left({D_{1,n}-2\sqrt{n}}\right){n^{1/6}}}{(\log n)^{2/3}}=\left(\frac{1}{4}\right)^{2/3}.
\end{align}
\end{theorem}

\medskip

\noindent
\textbf{Brownian LPP estimates:} Our proof of Theorem \ref{t:LPP} (and also Theorem \ref{t:LPPupper}) will consist of the following estimates at the level of Brownian LPP, which can then be used to prove the lower bound and the upper bound for the liminf in Theorem \ref{t:LPP}. The first estimate we need is the following sharp tail estimate from \cite{BBBK24} (the following is a restatement from \cite{BBBK24} using Proposition \ref{p:BLPP}).  

\begin{proposition}\cite[Theorem 1.4, Theorem 1.5]{BBBK24}
\label{p: lower tail estimates}For any $\vep>0$ there exist $n_{\vep}$ and $x_{\vep}$ such that for all $n \geq n_{\vep}$ and $x_{\vep} \leq x \leq n^{1/10}$ the following hold
\[
\exp \left(-\frac{1}{12}(1+\vep) x^3\right) \leq \P \left(\left(D_{1,n}-2 \sqrt{n}\right)n^{1/6} \leq -x \right) \leq \exp \left( -\frac{1}{12}(1-\vep)x^3\right).
\]
\end{proposition}

The next estimate shows strong correlation at the scale $k\ll n^{2/3}$.

\begin{proposition}\label{lem: liminf lower bound} 
Fix any $\vep,\vep'>0$ small enough. Then there exists $n_{\vep,\vep'}$ and $x_\vep$ such that for all $n\geq n_{\vep,\vep'}$, and $x_{\vep} \leq x \leq n^{1/20}$ we have  
\begin{align*}
    \P\l(\min\limits_{n\leq i\leq n+n^{\frac{2}{3}-\vep'}}\l(D_{1,i}-2\sqrt{i}\r)i^{1/6}\leq -x\r)\leq \exp \left(-\frac{1}{12}(1-\vep)x^3 \right).
\end{align*}
\end{proposition}

The following result is our main decorrelation estimate. 

\begin{proposition}
    \label{p:decorr}
    For all $\varepsilon_1,\varepsilon_2>0$ {small enough} and $u_{k}=\lfloor k^{3+\varepsilon_1}\rfloor,$ there exists $\delta=\delta(\varepsilon_1,\varepsilon_2)>0$ and $k_0=k_0(\varepsilon_1)>0$ and a sequence of independent random variables $X_{k}$ such that for all $k>k_0$
    \begin{equation}
    \label{e:approx}
        \P\l(\l | X_{k}-D_{1,u_k}\r |\ge \varepsilon_2 (u_k)^{-1/6}\r)\le \exp\l(-k^{\delta}\r). 
    \end{equation}
\end{proposition}

We can now complete the proof of Theorem \ref{t:LPP} using the estimates above. 

\begin{proof}[Proof of Theorem \ref{t:LPP}]
\textbf{Lower Bound:} If we take\footnote{Throughout our discussion we shall omit the floor/ceiling functions to avoid notational overload. Reader can easily check that this does not affect our arguments in any non-trivial way.} $n_k=k^{3-9\vep'}$ then for all large enough $k$ (depending on $\vep'$), we have $\l(k+1\r)^{3-9\vep'}\leq n_k+n_k^{\frac{2}{3}-\vep'}$. Using this fact it follows that one can choose $\vep'$ small enough (dependence on $\vep$ occurs here) along with $n_{\vep,\vep'}$ and apply Proposition \ref{lem: liminf lower bound} (by taking $x=-(4+\vep)^{1/3} (\log n_k)^{1/3}$), to obtain that for some $\delta>0$ for all $k,n$ sufficiently large as above
\[
\P \left(\min_{n_k \leq i \leq n_{k+1}}  \frac{\left({D_{1,i}-2\sqrt{i}}\right){i^{1/6}}}{(\log i)^{1/3}} \leq -(4+\vep)^{1/3}\right) \leq \frac{1}{k^{1+\delta}}
\]
Hence, by Borel-Cantelli lemma almost surely
 \begin{align}
\liminf_{n\rightarrow\infty}\frac{\left({D_{1,n}-2\sqrt{n}}\right){n^{1/6}}}{(\log n)^{1/3}} \geq -(4+\vep)^{1/3}.
\end{align}
The lower bound follows as $\vep$ was arbitrary. 

\medskip
\noindent
\textbf{Upper Bound:}
Notice that by \eqref{e:approx} and the Borel-Cantelli lemma it suffices to show that 
$$\liminf_{k\to \infty} \dfrac{u_k^{1/6}\l(X_{k}-2\sqrt{u_k}\r)}{(\log u_k)^{1/3}}\le -4^{1/3}$$
almost surely. 

For $\varepsilon>0$, Proposition \ref{p: lower tail estimates} implies that for $\varepsilon_1$ sufficiently small (depending on $\vep$), there is $\vep'$ depending on $\vep$ such that
$$\P\left(\dfrac{u_k^{1/6}\l(D_{1,u_k}-2\sqrt{u_k}\r)}{(\log u_k)^{1/3}}\le -(4-\varepsilon)^{1/3}\right) \geq \frac{1}{k^{1-\vep'}}$$
for all $k$ sufficiently large (depending on $\vep,\vep_1$). It follows from \eqref{e:approx} that for $k$ sufficiently large as above, 
$$\P\left(\dfrac{u_k^{1/6}(X_k-2\sqrt{u_k})}{(\log u_k)^{1/3}}\le -(4-2\varepsilon)^{1/3}\right) \ge \frac{1}{2k^{1-\varepsilon'}}$$
for all $k$ sufficiently large. The second Borel-Cantelli lemma now implies that  
$$\liminf_{k\to \infty} \dfrac{u_k^{1/6}\l(X_{k}-2\sqrt{u_k}\r)}{(\log u_k)^{1/3}}\le -(4-2\varepsilon)^{1/3}$$
almost surely. As $\varepsilon>0$ is arbitrary, we are done.
\end{proof}

The proof of Theorem \ref{t:LPPupper} follows from similar arguments; apart from Proposition \ref{p:decorr}, we need the following two results which are counterparts of Propositions \ref{p: lower tail estimates} and \ref{lem: liminf lower bound} respectively. 

\begin{proposition}\cite[Theorem 1.1, Theorem 1.3]{BBBK24}
\label{p: upper tail estimates}For any $\vep>0$ there exist $n_{\vep}, \gamma_{\vep}$ and $x_{\vep}$ such that for all $n \geq n_{\vep}$ and $x_{\vep} \leq x \leq \gamma_{\vep}n^{2/3}$ the following hold
\[
\exp \left(-\frac{4}{3}(1+\vep) x^{3/2}\right) \leq \P \left(\left(D_{1,n}-2 \sqrt{n}\right)n^{1/6} \geq x \right) \leq \exp \left( -\frac{4}{3}(1-\vep)x^{3/2}\right).
\]
\end{proposition}

\begin{proposition}\label{lemma: maximum of point to interval estimates}
    For any $\vep, \vep'>0$ small enough, there exist $n_{\vep,\vep'},x_\vep$ such that for all $n \geq n_{\vep,\vep'}$ and $x_{\vep} \leq x \leq n^{1/20} $
    \[
    \P \left(\max_{n \leq j \leq n+n^{\frac{2}{3}-\vep'}}\left(D_{1,j}-2 \sqrt{j}\right)j^{1/6} \geq x \right) \leq e^{-\frac{4}{3}(1-\vep)x^{3/2}}.
    \]
\end{proposition}

We remark that a stronger estimate than Proposition \ref{p: upper tail estimates} for a smaller range of $x$ was proved in \cite[Lemma 7.3]{PZ17} (also see \cite{LR10, L07,BX23}). 

Using these results we can now quickly complete the proof of Theorem \ref{t:LPPupper}. 

\begin{proof}[Proof of Theorem \ref{t:LPPupper}]
\textbf{Lower Bound:} As in the proof of the upper bound in Theorem \ref{t:LPP}, by \eqref{e:approx} and the Borel-Cantelli lemma it suffices to show that 
$$\limsup_{k\to \infty} \dfrac{u_k^{1/6}\l(X_{k}-2\sqrt{u_k}\r)}{(\log u_k)^{2/3}}\ge \l(\frac 14 \r)^{2/3}$$
almost surely. 

For $\varepsilon>0$, Proposition \ref{p: upper tail estimates} implies that for $\varepsilon_1$ (in the definition of $u_k$) sufficiently small (depending on $\vep$), there is some $\vep'$ depending on $\vep$ such that 
$$\P\left(\dfrac{u_k^{1/6}\l(D_{1,u_k}-2\sqrt{u_k}\r)}{(\log u_k)^{2/3}}\ge \l(\frac{1}{4}-\varepsilon\r)^{2/3}\right) \ge \frac{1}{k^{1-\vep'}}$$
for all $k$ sufficiently large (depending on $\vep,\vep_1$). It follows from \eqref{e:approx} that
$$\P\left(\dfrac{u_k^{1/6}(X_k-2\sqrt{u_k})}{(\log u_k)^{2/3}}\ge \l(\frac{1}{4}-2\varepsilon\r)^{2/3}\right) \ge \frac{1}{2k^{1-\vep'}}$$
for all $k$ sufficiently large.

The second Borel-Cantelli lemma now implies that  
$$\limsup_{k\to \infty} \dfrac{u_k^{1/6}\l(X_{k}-2\sqrt{u_k}\r)}{(\log u_k)^{2/3}}\ge \l(\frac{1}{4}-2\varepsilon\r)^{2/3}$$
almost surely. As $\varepsilon>0$ is arbitrary, we are done. 

\medskip
\noindent
\textbf{Upper Bound:} If we take $n_k=k^{3-9\vep'}$ then for all large enough $k$ (depending on $\vep'$), we have $\l(k+1\r)^{3-9\vep'}\leq n_k+n_k^{\frac{2}{3}-\vep'}$. Using this fact it follows that one can choose $\vep'$ small enough (dependence on $\vep$ occurs here) along with $k_{\vep,\vep'}$ and apply Proposition \ref{lemma: maximum of point to interval 
 estimates} (take $x=(\frac{1}{4}+\vep)^{2/3}(\log n_k)^{2/3}$), Borel-Cantelli lemma to obtain that almost surely
 \begin{align}
\limsup_{n\rightarrow\infty}\frac{\left({D_{1,n}-2\sqrt{n}}\right){n^{1/6}}}{(\log n)^{2/3}}<\l(\frac{1}{4}+\vep\r)^{2/3}
\end{align}
and the proof is completed as before. 
\end{proof}

\subsection{Law of fractional logarithms for the Plancherel growth process}
\label{s:lis}
We now describe our result about \emph{longest increasing subsequences} in a growing random permutation, answering the second question of Kalai \cite{Kalai_question}. We start by describing a construction of growing random permutations that corresponds to the \emph{Plancherel growth process}. 

For distinct $x_1,x_2,\ldots, x_n\in [0,1]$ let 
$x_{(1)}<x_{(2)}<\cdots < x_{(n)}$ denote the order statistics, i.e., $x_i$s reordered in the increasing order. Now, let $\sigma(x_1,x_2,\ldots, x_n)\in S_n$ denote the permutation of length $n$ such that $\sigma(i)=j$ where $j$ is such that $x_{i}=x_{(j)}$, i.e., $j$ is the rank of $x_i$ in the increasing order. We have the following standard result. 

\begin{lemma}
    \label{l:permutation}
    Let $U_1,U_2,\ldots$ be a sequence of i.i.d.\ random variables distributed uniformly on $[0,1]$. Let $\sigma_n=\sigma(U_1,U_2,\ldots, U_n)$. Then the sequence of random permutations $\{\sigma_n\}_{n\ge 1}$ is a Markov chain with uniform marginals on $S_n$ starting from $\sigma_1=(1)$ (the identity permutation of size $1$) and the conditional distribution of $\sigma_{n+1}$ given $\sigma_n$ is given as follows: 
    \begin{equation}
    \label{eq:1}\P(\sigma_{n+1}(n+1)=j\mid \sigma_n)=\frac{1}{n+1}\end{equation} for $j=1,2,\ldots, n+1$ and for $j\le n$, conditional on $\sigma_n$ and $\sigma_{n+1}(n+1)$ we have
\[
\sigma_{n+1}(j)= \begin{cases}
    \sigma_n(j) \quad \text{if } \sigma_{n+1}(n+1)> \sigma_n(j),\\
    \sigma_n(j)+1 \quad \text{otherwise}.
\end{cases}
\]
\end{lemma}

\begin{proof}

Given $\sigma_{n}$, we can recover $\sigma_{n-1}$. Indeed, for $i<n$ such that $\sigma_n(i)>\sigma_n(n)$, we have $\sigma_{n-1}(i)=\sigma_{n}(i)-1$. Else $\sigma_{n-1}(i)=\sigma_{n}(i)$. This shows that $\{\sigma_n\}_{n\geq 1}$ is Markov chain. This also gives the last part of the lemma, that is, given any $\sigma_n$ and $\sigma_{n+1}(n+1)$ we can recover $\sigma_{n+1}$. Also note that for any $n$, the distribution of $\sigma_n$ is uniformly distributed in $S_n$. Hence for any $\tau_n\in S_n$, we have that $\P\l(\sigma_{n+1}(n+1)=j \mbox{ and } \sigma_n=\tau_n\r)=\frac{1}{(n+1)!}$. This shows \eqref{eq:1}.
\end{proof}
 By the RSK correspondence, each $\sigma_n$ has an associated random Young diagram $\mu_n$, where $\mu_n=(\mu_n^1,\mu_n^2,\ldots)$ is an integer partition of $n$ with $\mu_n^{i}\ge \mu_n^{(i+1)}$ for all $i$ and $\sum_{i} \mu_n^i=n$ (here $\mu_n$ is the common shape of the two Young tableaux to which $\sigma_n$ maps via the RSK correspondence). As $\sigma_n$ has uniform distribution, the marginal distribution of $\mu_n$ is the Plancherel measure on Young diagrams with $n$ boxes. This process of growing random Young diagrams $\{\mu_n\}_{n\ge 1}$ is known as the \emph{Plancherel growth process} \cite{VerKer77}.

Recall that a subsequence $1\le i_1<i_2<\cdots <i_k \le n$ is called an increasing subsequence of $\sigma\in S_n$ if $\sigma(i_1)<\sigma(i_2)<\cdots <\sigma(i_k)$. It is a well known result that $L_n$, the length of the longest increasing subsequence of $\sigma_n$ equals $\mu_n^1$, the length of the top row of $\mu_n$. The famous Ulam's problem \cite{ulam1961monte} asks about asymptotics of $L_n$. We refer the reader to \cite{Romik_2015} for more details on these classical developments.

As already mentioned, $L_n$, properly centered and scaled, converges weakly to the GUE Tracy-Widom disctribution. Our final main result establishes a law of fractional logarithm (similar to Theorems \ref{thm: liminf constant} and \ref{t:LPPupper}) for $L_n$ and answers the corresponding question of Kalai \cite{Kalai_question}. 

\begin{theorem}
\label{t: LFL for LIS}
We have, for $L_n$ as above
\begin{enumerate}[label=(\roman*), font=\normalfont]
   \item 
   \[\limsup_{n \rightarrow \infty} \frac{L_{n}-2\sqrt{n}}{n^{1/6}(\log n)^{2/3}}=\l( \frac 18 \r)^{2/3} \text{ a.s. }\]
   \item \[\liminf_{n \rightarrow \infty} \frac{L_{n}-2\sqrt{n}}{n^{1/6}(\log n)^{1/3}}=-2^{1/3} \text{ a.s. } \]   
\end{enumerate}
\end{theorem}

Notice that the scaling functions in the above theorem are the same as in the case of GUE minor process. This is to be expected as these correspond to the tail exponents in the GUE Tracy-Widom distribution. However, the limiting constants are different, because the decorrelation happens at a different, albeit still sub-linear, scale. 

We shall, once again, translate this question to a question in last passage percolation, this time to the \emph{Poissonian last passage percolation}. The proof of Theorem \ref{t: LFL for LIS} is going to be similar to the proofs for the GUE minor process. 

\medskip 

\noindent
\textbf{Poissonian last passage percolation:}
Consider a rate $1$ Poisson point process $\Pi$ on $\R^2$. For $u,v\in \R^2$ with $u$ co-ordinate wise smaller than $v$, let $\cl_{u,v}$, the last passage time from $u$ to $v$, denote the maximum number of Poisson points on an up/right path from $u$ to $v$ (for concreteness we only allow piecewise linear paths with where each linear piece connects two points in $\Pi$, or a point in $\Pi$ to $u$ or $v$). When $u=(0,0)$ and $v=(a,b)$ we shall denote the corresponding last passage time by $\cl_{a,b}$. 

The connection between the Poissonian LPP and (a Poissonized version of) Ulam's problem is well-known. Let $R$ be a rectangle with $u$ and $v$ being the bottom left and top right corners respectively. One can create a permutation $\Sigma_R$ of size $N_R$ from $\Pi$ restricted to any rectangle $R$, where $N_R$ denotes the number of points in $\Pi \cap R$, by considering the ranks of the $x$-coordinates of points ordered according to increasing $y$ coordinate (as in the definition of $\sigma(x_1,x_2,\ldots,x_n)$ above). It is standard (and easy to see) that conditional on $N_R$, the permutation $\Sigma_R$ is uniform in $S_{N_R}$ and the length of its longest increasing subsequence is $\cl_{u,v}$.

To cleanly state a distributional equality between the random permutations $\{\sigma_n\}$ from Lemma \ref{l:permutation} we shall introduce a continuous time version of the process $\sigma_n$. In the set up of Lemma \ref{l:permutation}, let $T_1,T_2,\ldots$ denote a sequence of i.i.d. $\exp(1)$ random variables independent of the sequence $\{U_i\}$. Let $N(t)$ denote the Poisson process on $\R$ with waiting times $T_i$, i.e.,
$$ N(t)=\max\{n: T_1+T_2+\cdots +T_n\le t\}.$$
Define a continuous time process of growing random permutations by setting 
$$\widehat{\sigma}_t=\sigma_{N(t)}.$$
Let $\widehat{L}_t$ denote the length of the longest increasing subsequence of $\widehat{\sigma}_t$. The next proposition allows us to prove Theorem \ref{t: LFL for LIS} in the language of Poissonian LPP.

\begin{proposition}
    \label{p:PlanchereltoPoisson}
    Let $R_t=[0,1]\times[0,t]$. Then for the permutations $\Sigma_{R_{t}}$ described above, we have
    $$\{\Sigma_{R_t}\}_{t>0}\overset{d}{=} \{\widehat{\sigma}_t\}_{t>0}.$$
    As a consequence, we have
    $$\{\cl_{1,t}\}_{t>0}\overset{d}{=}\{\widehat{L}_t\}_{t>0}.$$
\end{proposition}

\begin{proof}
   The main observation is that for $U_i$ and $T_i$ as above the point process $\Pi'=\{(U_1,T_1), (U_2,T_1+T_2), (U_3, T_1+T_2+T_3), \ldots\}$ has the same law as a rate $1$ Poisson point process on $[0,1]\times [0,\infty)$. Indeed, one can see this by first noticing that $T_i$s are inter arrival times of a rate $1$ Poisson process and then using Poisson thinning to show that for disjoint rectangles $R'_i$ in $[0,1]\times [0,\infty)$ the number of points in $R'_i\cap \Pi'$ are independent Poisson random variables with parameters equal to the areas of $R'_i$. The first statement in the proposition now follows from the definitions. The second assertion is a consequence of the first one.   
    \end{proof}

We shall prove the following result, which, by Proposition \ref{p:PlanchereltoPoisson} and a simple de-Poissonization argument, will imply Theorem \ref{t: LFL for LIS}.

\begin{proposition}
\label{l:LFL for sequence}
We have 
\begin{enumerate}[label=(\roman*), font=\normalfont] 
    \item 
    \[\limsup_{n \rightarrow \infty} \frac{\cl_{1,n}-2\sqrt{n}}{n^{1/6}(\log n)^{2/3}} = \l( \frac 18 \r)^{2/3} \text{ a.s. }\]
    \item \[\liminf_{n \rightarrow \infty} \frac{\cl_{1,n}-2\sqrt{n}}{n^{1/6}(\log n)^{1/3}} = -\l( 2 \r)^{1/3} \text{ a.s. }
    \]
    \end{enumerate}
    \end{proposition}

It will be clear from the proof that the conclusion of Proposition \ref{l:LFL for sequence} remains true for $\mathcal{L}_{1,t}$ when $t\to \infty$ through reals, not just through integers. However, Proposition \ref{l:LFL for sequence} will be sufficient for us to prove Theorem \ref{t: LFL for LIS}.

We shall now complete the proof of Theorem \ref{t: LFL for LIS} using Proposition \ref{l:LFL for sequence}. 

\begin{proof}[Proof of Theorem \ref{t: LFL for LIS}]
    Consider the set-up of Poissonian LPP as described above. Let $P_n$ be the number of points in the strip $[0,1]\times [0,n]$. By Poisson deviation bounds and the Borel-Cantelli lemma we have that almost surely for all $n\ge n_0$ (for some random $n_0$) $P_{\lfloor n-n^{5/8}\rfloor}\leq n\leq P_{\lceil n+n^{5/8}\rceil}$. This implies that almost surely for large enough $n$ 
\begin{align}
     \label{eq:depossion}\mathcal{L}_{1,\lfloor n-n^{5/8}\rfloor}\leq L_{n}\leq \mathcal{L}_{1,\lceil n+n^{5/8}\rceil}.
\end{align}
Notice that there is nothing special about the exponent ${5/8}$ any number strictly between $1/2$ and $2/3$ will satisfy our purposes. 

We shall now prove the $\limsup$ statement in Theorem \ref{t: LFL for LIS} by using the Proposition \ref{l:LFL for sequence}, (i). The proof of the second statement in Theorem \ref{t: LFL for LIS} is almost identical using Proposition \ref{l:LFL for sequence}, (ii) and is omitted. 

For convenience of notation let us set $n_+=\lceil n+n^{5/8} \rceil$. It follows from \eqref{eq:depossion} that for all $n$ sufficiently large 

$$ \frac{L_n-2\sqrt{n}}{n^{1/6}(\log n)^{2/3}}\le b_n \frac{\cl_{1,n_{+}}-2\sqrt{n_{+}}}{n_{+}^{1/6}(\log n_+)^{2/3}} + \frac{a_n}{n^{1/6}(\log n)^{2/3}}$$
where $a_n=2\sqrt{n_{+}}-2\sqrt{n}=o(n^{1/6})$ (this is where we use $5/8<2/3$) and 
$$b_n:=\frac{n_+^{1/6}(\log n_+)^{2/3}}{n^{1/6}(\log n)^{2/3}}\to 1$$
as $n\to \infty$. It therefore follows from Proposition \ref{l:LFL for sequence}, (i) that almost surely
$$\limsup_{n \rightarrow \infty} \frac{L_{n}-2\sqrt{n}}{n^{1/6}(\log n)^{2/3}} \le \l( \frac 18 \r)^{2/3}.$$

To show that almost surely
$$\limsup_{n \rightarrow \infty} \frac{L_{n}-2\sqrt{n}}{n^{1/6}(\log n)^{2/3}} \ge \l( \frac 18 \r)^{2/3},$$
we need to almost surely exhibit a (possibly random) subsequence along which $\frac{L_{n}-2\sqrt{n}}{n^{1/6}(\log n)^{2/3}}$ converges to $\l( \frac 18 \r)^{2/3}$. By Proposition \ref{l:LFL for sequence}, (i), almost surely there exists a subsequence $n_k$ such that 

$$\lim_{k \rightarrow \infty} \frac{\cl_{1,n_k}-2\sqrt{n_k}}{n_k^{1/6}(\log n_k)^{2/3}} = \l( \frac 18 \r)^{2/3}.$$ 

Writing $h_k=P_{n_k}$, by definition, $\cl_{1,n_k}=L_{h_k}$ and therefore and it follows that 
$$\frac{L_{h_k}-2\sqrt{h_k}}{h_k^{1/6}(\log h_k)^{2/3}}=d_{k} \frac{\cl_{1,n_k}-2\sqrt{n_k}}{n_k^{1/6}(\log n_k)^{2/3}}+ \frac{c_k}{h_k^{1/6}(\log h_k)^{2/3}}$$
where 
$d_k=\frac{n_k^{1/6}(\log n_k)^{2/3}}{h_k^{1/6}(\log h_k)^{2/3}}$ and $c_k=2\sqrt{n_k}-2\sqrt{h_k}$. 

Arguing as in \eqref{eq:depossion} it follows that almost surely for all $n$ sufficiently large $n-n^{5/8}\le P_n \le n+n^{5/8}$ and therefore $d_k\to 1$ and $c_k=o(h_k^{1/6})$ as $k\to \infty$. It therefore follows that 
$$\lim_{k\to \infty} \frac{L_{h_k}-2\sqrt{h_k}}{h_k^{1/6}(\log h_k)^{2/3}} = \l( \frac 18 \r)^{2/3},$$
as required. 
\end{proof}

Like the Brownian LPP case (see the remark after Theorem \ref{t:LPP}), if we move along the diagonal direction (i.e., consider $\cl_{t,t}$ instead of $\cl_{1,t}$) we will get a different scaling. In particular, as we discussed in the Brownian LPP case we will obtain a \textit{law of iterated logarithm} in the diagonal direction (for the limsup result see \cite{Su20}).

The proof of Proposition \ref{l:LFL for sequence}, which uses arguments rather similar to the ones in the proofs of Theorems \ref{t:LPP} and \ref{t:LPPupper}, is given in Section \ref{s:lisfull}. Recall that these proofs rely on (a) the sharp tail estimates for the last passage times and (b) the decorrelation estimate. The sharp tail estimates for the Poissonian LPP are available from \cite{LM01,LMS02}, while the decorrelation estimate is proved in Proposition \ref{p: Poisson decorr} using arguments similar to the proof of Proposition \ref{p:decorr}. The only difference in this case is that the decorrelation happens at scale $n^{5/6}$ (compared to the scale $n^{2/3}$ in the Brownian LPP case) since the typical vertical transversal fluctuation of geodesics from $(0,0)$ to $(1,n)$ in Poissonian LPP is of the order $n^{5/6}$. We postpone further details to Section \ref{s:lisfull}. 

\subsection{Two $0-1$ laws}
As mentioned before, we also establish two general $0-1$ laws for the sequence $D_{1,n}$ and for the function $\cl_{1,t}$. Although our proofs of Theorem \ref{t:LPP} and Theorem \ref{t:LPPupper}, Proposition \ref{l:LFL for sequence} do not rely on the following facts, these results are interesting in its own sake and we shall provide the proof at the end of next section.
\begin{proposition}
    \label{p:01} Let $f:\R_+\rightarrow \R_+$ be a function which satisfies (i) $\lim\limits_{t \rightarrow \infty} f(t)=\infty$ and (ii) $\lim\limits_{t \rightarrow \infty} \frac{f(t)}{f(t-1)}=1.$ 
    \begin{enumerate}[label=(\roman*), font=\normalfont]
    \item
    Then, the limsup and liminf of  
$$\frac{\left({D_{1,n}-2\sqrt{n}}\right){n^{1/6}}}{f(n)}$$
   are constants almost surely.
   \item Then the limsup and liminf of 
   $$\frac{\left({\cl_{1,t}-2\sqrt{t}}\right)}{f(t)t^{1/6}}$$ are constants almost surely.
   \end{enumerate}
    
\end{proposition}
\begin{remark}Note that in the first statement both $(\log t)^{2/3}$ and $(\log t)^{1/3}$ satisfy the conditions of the proposition. Also, if $f(n) \gg (\log n)^{2/3}$ then by Theorem \ref{t:LPPupper} and Theorem \ref{t:LPP} both the $\limsup$ and $\liminf$ are zero, and if $f$ is such that $f(n) \ll (\log n)^{1/3}$ then both the $\limsup$ and $\liminf$ are infinite. However, if $f(n)$ is such that it oscillates between $(\log n)^{2/3}$ and $(\log n)^{1/3}$ then the result in Proposition \ref{p:01}, (i) does not follow from Theorem \ref{t:LPPupper} or Theorem \ref{t:LPP}. For example, let us define the function $f(x)$ such that
\[
f(x)= \begin{cases}
    ( \log x)^{2/3} \text{ when } x=(2n)^3\\
    (\log x)^{1/3} \text{ when } x=(2n+1)^3\\
    \text{linear in between}.
\end{cases}
\]
Note that $f(n)$ satisfies the conditions in Proposition \ref{p:01} and hence, the $\limsup$ and the $\liminf$ are constants  after dividing by $f(n)$, but this does not follow from Theorem \ref{t:LPPupper} and Theorem \ref{t:LPP}. Similar remark is true for the second statement as well. 
\end{remark}

\subsection*{Organisation} Section \ref{s:blpp} provides the remaining proofs in the Brownian LPP case, i.e., the proofs of Propositions \ref{lem: liminf lower bound}, \ref{p:decorr} and \ref{lemma: maximum of point to interval estimates}. Section \ref{s:lisfull} completes the arguments in the Poissonian LPP case by proving Proposition \ref{l:LFL for sequence}. Section \ref{s:01} proves Proposition \ref{p:01}.

\subsection*{Acknowledgements}  This work is partly supported by the DST FIST program-2021[TPN-700661].  JB is supported by scholarship from Ministry of Education (MoE). RB is partially supported by a MATRICS grant (MTR/2021/000093) from SERB, Govt.\ of India, DAE project no.\ RTI4001 via ICTS, and the Infosys Foundation via the Infosys-Chandrasekharan Virtual Centre for Random Geometry of TIFR. SB is supported by scholarship from National Board for Higher Mathematics (NBHM) (ref no: 0203/13(32)/2021-R\&D-II/13158).

\section{Proofs of the Brownian LPP estimates}
\label{s:blpp}
Before starting with the proofs of the main propositions we need some basic estimates. 

\subsection{Auxiliary estimates}
\label{s: auxiliary estimates}
In this subsection we record some facts about Brownian LPP that we use throughout our proofs. First we need the existence and uniqueness of maximum passage time attaining paths. For any $x,y \in \R, i,j \in \N$ with $x\leq y$ and $i \leq j$ and $x=t_0 \leq t_1 \leq \cdots \leq t_{j-i+1}=y$ the collection of the intervals $[t_{k-1},t_{k}] \times \{i+k-1\}$ is called a path from $(x,i)$ to $(y,j).$

\begin{proposition}\cite[Lemma B.1]{hammond19} For a fixed pair $(x,i),(y,j)$ with $y \geq x$ and $j \geq i$ there exists an almost surely unique path which attains $D_{(x,i) \rightarrow (y,j)}.$
\end{proposition}
Maximum attaining paths are called \textit{geodesics}. The geodesic to $(0,1)$ to $(t,n)$ will be denoted by $\Gamma_{t,n}.$ For ordered pair of vertices $u,v \in \R \times \N$, let $ \Gamma_{u,v}$ denote the a.s. unique geodesic from $u$ to $v$. Our first observation is about ordering of geodesics which is a direct consequence of \cite[Lemma 5.7]{Ham20}.
\begin{observation}
\label{o: ordering}
    Let $n_2 > n_1 \geq 0$ and consider the geodesics $\Gamma_{1,n_1}$ and $\Gamma_{1,n_2}.$ Then for all $n \leq n_1$ if $(x,n) \in \Gamma_{1, n_2}$ and $(y,n) \in \Gamma_{1,n_1}$ then $x \leq y.$
\end{observation}
\begin{proof} Note that the restriction of the geodesic $\Gamma_{1,n_2}$ up to the horizontal line $y=n_1$ is a geodesic between $(0,1)$ and its end point. Now, by the ordering of geodesics proved in \cite[Lemma 5.7]{Ham20} the observation follows.
\end{proof}
For $(t,n) \in [0, \infty) \times \N $, the geodesic $ \Gamma_{t,n}$ can be thought of a random step function (possibly multi valued) from $[0,t]$ to $\left \{1,2,\dots, n \right \}$ defined as follows.
\[
\Gamma_{t,n}(x)=i \text{ if } (x,i) \in \Gamma_{t,n}.
\] We define another random step function $\Gamma'_{t,n}:[0,1] \rightarrow \{ 1,2, \dots, n\}$ as $\Gamma'_{t,n}(x)=\Gamma_{t,n}(tx).$ As a consequence of Brownian scaling we have the following scaling of last passage time and geodesics in Brownian LPP.

\begin{observation}[Scaling of Brownian LPP]
\label{p: scaling}We have 
\begin{enumerate}[label=(\roman*), font=\normalfont]
\item For all $t >0$ 
\[
\left \{\frac{D_{t,n}}{\sqrt{t}} \right\}_{n \geq 1}\overset{d}{=} \left \{D_{1,n} \right \}_{n \geq 1}.
\]
\item As random step functions
\[
\Gamma'_{t,n} \overset{d}{=}\Gamma_{1,n}.
\]
\end{enumerate}
\end{observation}

\begin{proof}(i) This is a direct consequence of Brownian scaling. By the Brownian scaling the finite dimensional distributions $\left \{D_{1,i} \right \}_{1 \leq i \leq n}$ and $\left \{\frac{D_{t,i}}{\sqrt{t}} \right \}_{1 \leq i \leq n}$ are same and hence, as an infinite random vector also $\left \{\frac{D_{t,n}}{\sqrt{n}} \right \}_{n \geq 1}$ and $\left \{D_{1,n} \right\}_{n \geq 1}$ have same distributions.\\
(ii) The second part follows from the first part and by observing that after scaling, any partition $0 \leq t_0 \leq t_1 \leq t_2 \leq \dots \leq t_n=t$ is in one-one correspondence with $0 \leq \frac{t_0}{t} \leq \frac{t_1}{t} \leq \dots \leq \frac{t_n}{t}=1$ which is a partition of $[0,1].$
\end{proof}

Observation \ref{p: scaling} is not new; it has been used before to establish the skew-invariance (sometimes also referred to as shear-invariance) of the \emph{directed landscape}, the space-time scaling limit of Brownian LPP; see e.g.\ \cite{DOV22}. In the Brownian LPP set-up this has also been used in \cite{BGH21}.

For any $ 0 \leq t \leq 1$, let $\Gamma_{1,n}(t):=\{i \in \{1,2, \dots n\}: (t,i) \in \Gamma_{1,n} \}$. The final lemma we need is about the transversal fluctuation and the local transversal fluctuation of the geodesics in the vertical direction. For $n \geq 1$ we define $\mathsf{TF}_n$ as follows.
\[
\mathsf{TF}_n:=\max_{0 \leq t \leq 1} \left \{\max_{i \in \Gamma_{1,n}(t)}|i-tn|. \right \}
\]
\begin{lemma} 
    \label{lemma: transversal fluctuation}We have\begin{enumerate}[label=(\roman*), font=\normalfont]
        \item There exist $c, n_0,\ell_0>0$ such that for all $n \geq n_0, \ell_0 \leq \ell\leq n^{1/10}$
        \[\P \left( \mathsf{TF}_n \geq \ell n^{2/3} \right) \leq \exp(-c\ell^3).
        \]
        \item Let $n_0,\ell_0$ be as above and let $n_1= \max \{\ell_0^{10}, n_0 \}$. Then there exists $c'>0$ such that for all $n \geq n_1, \ell_0 \leq \ell\leq (tn)^{1/10}$ and for all $t \in [0,1]$ with $tn \geq n_1$
        \[\P \left(\max_{i \in \Gamma_{1,n}(t)}|i-tn| \geq \ell (tn)^{2/3}\right) \leq \exp(-c'\ell^3).\]
    \end{enumerate}
\end{lemma}

We shall provide a proof of this lemma at the end of this section. Lemma \ref{lemma: transversal fluctuation}, (i) is essentially same as \cite[Corollary 1.5]{GanHam23} but \ref{lemma: transversal fluctuation}, (ii) appears to be new for Brownian LPP. Analogous statements have been proved in the exponential LPP, however, see e.g.\ \cite{BSS17B,BBB23}. We shall provide an argument along the same lines for the sake of completeness. Although the range of $\ell$ in the lemma is sufficient for our purpose, we expect the lemma to be true for all $\ell \leq n^{1/3}$. Indeed it is the case in case of exponential LPP (see \cite{BGZ21, BBB23}).

\subsection{Proofs of Proposition \ref{lem: liminf lower bound} and Proposition \ref{lemma: maximum of point to interval estimates}}
\begin{proof}[Proof of Proposition \ref{lem: liminf lower bound}] A stronger version of the proposition was proved for exponential LPP case in \cite[Lemma 3.9]{BB24}. Our proof follows a similar idea. Choose any small $\vep'$. By Observation \ref{p: scaling}, (i) it suffices to prove 
\begin{equation}
\label{eq: scaled liminf lower bound}
    \P\l(\min\limits_{n\leq i\leq n+n^{\frac{2}{3}-\vep'}}\frac{\l(D_{n,i}-2\sqrt{ni}\r)i^{1/6}}{\sqrt{n}}\leq -x\r)\leq \exp \left(-\frac{1}{12}(1-\vep)x^3\right).
\end{equation}

Note that for $n \leq i\leq  n+n^{\frac{2}{3}-\vep'}$, it suffices to consider the event 
\begin{align*}
    \mathcal{A}:= \left\{\min\limits_{n\leq i\leq n+n^{\frac{2}{3}-\vep'}}\frac{\l(D_{n,i}-2\sqrt{ni}\r)}{n^{1/3}}\leq -x\l(1-\frac \vep2 \r)\right\}
\end{align*}
and show that $\P\l(\mathcal{A}\r)\leq\exp \left(-\frac{1}{12}(1-\vep)x^3 \right)$. Let $n_{\vep'}=n^{1-\frac 32 \vep'}.$ Define the events (see Figure \ref{fig: GUE liminf lower bound})
\begin{figure}[t!]
\includegraphics[width=8 cm]{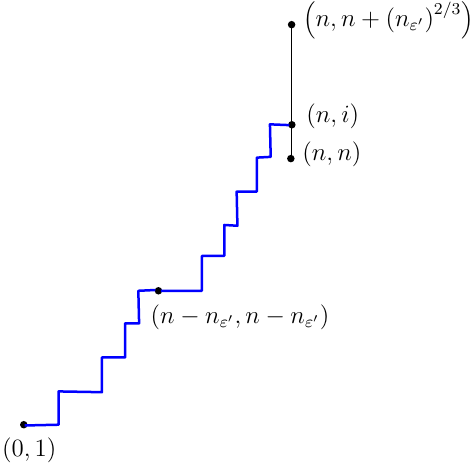}
    \caption{To prove Proposition \ref{lem: liminf lower bound}, by scaling it is enough to prove \eqref{eq: scaled liminf lower bound}. To do this we observe that for passage time between $(0,1)$ and $(n,i)$ to be small for some $i$ one of the following events must happen. Either the passage time between $(0,1)$ to $(n-n_{\vep'}, n-n_{\vep'})$ is small or the passage time between $(n-n_{\vep'}, n-n_{\vep'})$ and $(n,i)$ is small. We use the tail estimate from Proposition \ref{p: lower tail estimates} to find the desired upper bound for the first event. For each $i$ we again apply the same proposition to get an exponentially small (in $n$) upper bound for the later events. Finally, taking a union bound over $i$ gives us the upper bound in Proposition \ref{lem: liminf lower bound}.  } 
    \label{fig: GUE liminf lower bound}
\end{figure}
\begin{align*}
    \mathcal{B}&:=\left\{D_{n-n_{\vep'},n-n_{\vep'}}\leq 2(n-n_{\vep'})-x\l(1-\frac{3\vep}{4} \r)n^{1/3}\right\},\\
    \mathcal{C}&:=\left\{\min\limits_{n_{\vep'} \leq j\leq n_{\vep'}+\left(n_{\vep'}\right)^{2/3} }\wih{D}_{n_{\vep'}, j}\leq 2\sqrt{n(n-n_{\vep'}+j)}-2n+2n_{\vep'}-\frac\vep4x n^{1/3}\right\},
\end{align*}
where $\wih{D}_{n_{\vep'}, i}$ is the passage time from $(n-n_{\vep'},n-n_{\vep'})$ to $(n, n-n_{\vep'}+j)$. Note that $\P\l(\mathcal{C}\r)$ remains same if we replace $\wih{D}_{n_{\vep'}, j}$ by $D_{n_{\vep'}, (j+1)}$.

Clearly $\P\l(\mathcal{A}\r)\leq \P\l(\mathcal{B}\r)+\P\l(\mathcal{C}\r)$ (see Figure \ref{fig: GUE liminf lower bound}). Using the tail estimate from Proposition \ref{p: lower tail estimates}, we have some $n_{\vep,\vep'}, x_{\vep}$ such that for all $n\geq n_{\vep,\vep'}$ and $x_{\vep} \leq x \leq n^{1/10}$ we get $\P\l(\mathcal{B}\r)\leq \exp \left(-\frac{1}{12}x^3(1-\vep) \right)$.

For the event $\mathcal{C}$ observe that in the range $n_{\vep'}\leq j\leq n_{\vep'}+n^{\frac{2}{3}-\vep'}$, we have
\begin{align*}
    2\sqrt{n(n-n_{\vep'}+j)}-2n +2n_{\vep'}=2\sqrt{(n_{\vep'})j}+O(n^{1/3}).
\end{align*}
This is because let $j=n_{\vep'}+x (n_{\vep'})^{2/3}$ for some $0 \leq x \leq 1.$ Then a Taylor series expansion gives us 
\begin{align*}
    &2\sqrt{n(n-n_{\vep'}+j)}-2n +2n_{\vep'}=2\sqrt{n(n+x(n_{\vep'})^{2/3})}-2n+2n_{\vep'}=2n_{\vep'}+xn^{\frac 23 -\vep'}+O(n^{1/3}).
\end{align*}
Also, 
\begin{align*}
    &2\sqrt{n_{\vep'}j}=2 \sqrt{n_{\vep'}(n_{\vep'}+x(n_{\vep'})^{2/3})}=2n_{\vep'}+x(n_{\vep'})^{2/3}+O(n^{1/3})=2n_{\vep'}+xn^{\frac 23-\vep'}+O(n^{1/3}).
\end{align*}
Hence $\P\l(\mathcal{C}\r)\leq\P\l(\mathcal{C}'\r) $, where
\begin{align*}
    \mathcal{C}':=\left\{\min\limits_{n_{\vep'}\leq j\leq n_{\vep'}+\left(n_{\vep'} \right)^{2/3}}{D}_{n_{\vep'},(j+1)}\leq 2\sqrt{n_{\vep'}.j}-\frac{\vep}{8}x n^{1/3}\right\}.
\end{align*}

Again using the tail estimate from Proposition \ref{p: lower tail estimates}
\[
\P \left( {D}_{n_{\vep'},(j-1)}\leq 2\sqrt{n_{\vep'}.j}-\frac{\vep}{8}x n^{1/3}\right)\leq \exp\l(-c_{\vep}x^3n^{ \frac 32 \vep'}\r).
\]
and taking union bound over $j$, we have some $n_{\vep,\vep'}$ such that for all $n\geq n_{\vep,\vep'}$ we get $\P\l(\mathcal{C}'\r) \ll \exp \left(-\frac{1}{12}(1-\vep)x^3 \right)$. 
This gives the desired upper bound for $\P\l(\mathcal{A}\r)$.
\end{proof}

\begin{proof}[Proof of Proposition \ref{lemma: maximum of point to interval estimates}]

A stronger version of this proposition for exponential LPP case was proved in \cite[Lemma 3.3]{BB24}. Our proof follows a similar idea. By Observation \ref{p: scaling}, (i) it is enough to prove that for any $\vep, \vep'>0$ there exist $n_{\vep,\vep'}, x_{\vep}$ such that for all $n \geq n_{\vep, \vep'}$ and $x_{\vep} \leq x \leq n^{1/20}$
\[
\P \left(\max_{n \leq j \leq n+n^{\frac{2}{3}-\vep'}}\left(D_{n,j}-2 \sqrt{nj}\right)j^{1/6} \geq \sqrt{n}x\right) \leq \exp\left(-\frac{4}{3}(1-\vep) x^{3/2}\right).
\]
Also observe that for large enough $n$
\begin{align*}
& \P \left(\max_{n \leq j \leq n+n^{\frac{2}{3}-\vep'}}\left(D_{n,j}-2 \sqrt{nj}\right)j^{1/6} \geq x \sqrt{n} \right)\\
& \leq \P \left(\max_{n \leq j \leq n+n^{\frac{2}{3}-\vep'}}\left(D_{n,j}-2 \sqrt{nj}\right) \geq x n^{1/3}\l(1-\frac {\vep}{2}\r) \right).
\end{align*}
 \begin{figure}[t!]
 \includegraphics[width=10 cm]{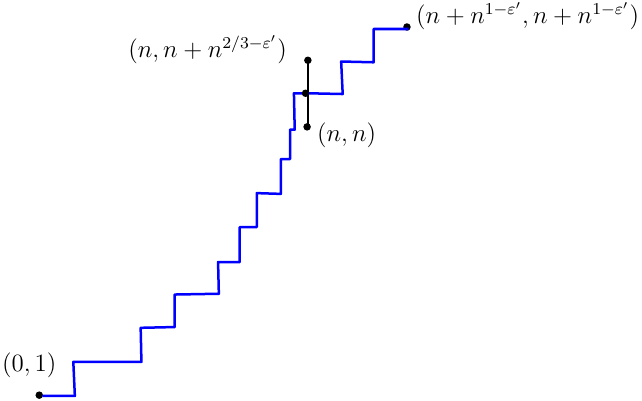}
    \caption{To prove Lemma \ref{lemma: maximum of point to interval estimates} we choose the point $(n+n^{1-\vep'},n+n^{1-\vep'})$. We want to find an upper bound for the event that there is $n \leq j \leq n+n^{\frac{2}{3}-\vep'}$ such that the centred and scaled passage time to $(n,j)$ is large. To do this we consider the two events where $D_{n+n^{1-\vep'},n+n^{1-\vep'}}$ is large (this is the event $\mathcal{A}$) and for all $j,$ the minimum passage time between $(n,j)$ and $(n+n^{1-\vep'},n+n^{1-\vep'})$ is not too small (this is the event $\mathcal{B}$). Note that if  there is $n \leq j \leq n+n^{\frac{2}{3}-\vep'}$ such that the centred and scaled passage time to $(n,j)$ is large and $\cb$ happens then $\ca$ also happens. We apply the upper bound for upper tail from Proposition \ref{p: upper tail estimates} to find the desired upper bound for $\P(\ca)$. By applying Proposition \ref{p: lower tail estimates} and applying a union bound we can make $\P(\mathcal{B}^c)$ is arbitrarily small.}
    \label{fig:point to interval(re)}
\end{figure}
We consider the point $$v=\left(n+n^{1-\vep'},n+n^{1-\vep'}\right).$$ 
For all $n \leq j \leq n+ n^{\frac{2}{3}-\vep'}$ (see Figure \ref{fig:point to interval(re)})
\[
D_{n+n^{1-\vep'}, n+n^{1-\vep'}} \geq D_{n,j}+D_{(n,j) \rightarrow v}.
\]
We consider the following two events. 
\begin{align*}
& \ca:= \left \{D_{n+n^{1-\vep'}, n+n^{1-\vep'}}-2 \left(n+n^{1-\vep'}\right) \geq xn^{1/3}\left (1- \vep \right) \right \}.\\
&\cb:=\left \{\min_{n \leq j \leq n+n^{\frac{2}{3}-\vep'}}\left(D_{(n,j) \rightarrow v}-2 \sqrt{n^{1-\vep'}(n+n^{1-\vep'} -j)}\right) \geq -\frac{\vep}{4}xn^{1/3}\right \}.
\end{align*}
Observe that by arguing similarly as before for all $n \leq j \leq n+n^{\frac{2}{3}-\vep'}$ there exists a constant $C>0$ independent of all the other parameters such that 
\[
2 \sqrt{nj}+2 \sqrt{n^{1-\vep'}(n+n^{1-\vep'} -j)} \geq 2(n+n^{1-\vep'} ) -Cn^{1/3}.
\]
 So, we can choose $x$ and $n$ sufficiently large depending on $\vep,\vep',C$ such that 
\[
 \left \{\max_{n \leq j \leq n+n^{\frac{2}{3}-\vep'}}\left(D_{n,j}-2 \sqrt{nj}\right) \geq x n^{1/3}\l( 1-\frac \vep2\r)\right \}\cap \cb \subset\ca  .
\]
We find upper bounds for $\P \left(\ca \right)$ and $\P \left(\cb^c \right)$. First observe that by Proposition \ref{p: upper tail estimates} for $n$ sufficiently large depending on $\vep,\vep'$ and $x$ sufficiently large depending on $\vep$ with $x \leq n^{1/20}$ we have
\[
\P \left(\ca \right) \leq \frac{1}{2}\exp \left(-\frac{4}{3} \l(1-\vep\r)x^{3/2} \right).
\]

Next we wish to find an upper bound for $\P \left( \cb^c\right).$ Note that for each $n \leq j \leq n+ n^{\frac{2}{3}-\vep'}$ we have again by Proposition \ref{p: lower tail estimates} for $x,n$ sufficiently large depending on $\vep'$ and $x \leq n^{1/20}$ we have,
\[
\P \left(D_{(n,j) \rightarrow v}-2 \sqrt{n^{1-\vep'}(n+n^{1-\vep'} -j)} \leq -\frac{\vep}{4}xn^{1/3}\right) \leq \exp \left(-c_{\vep} n^{\vep'} x^3 \right).
\]
Now, by taking a union bound over all $j$ we get 
\[
\P \left(\cb^c \right) \leq n\exp \left(-c_{\vep} n^{\vep'}x^3 \right).
\]
Thus, one can choose $n$ sufficiently large depending on $\vep, \vep'$ such that 
\[
\P \left(\cb^c \right) \ll \frac{1}{2}\exp\left(-\frac{4}{3}(1-\vep)x^{3/2} \right).
\]
This completes the proof.   
\end{proof}

\subsection{Proof of Proposition \ref{p:decorr}}
To prove Proposition \ref{p:decorr} we do the following construction which is similar to the construction in the proof of \cite[Theorem 3.8]{BB24}, where a limit theorem for the minima of Airy$_2$ process over growing intervals was proved. 

For $n\in \N$, consider the vertical intervals $I_{n}=\{1\}\times \l[n-n^{\frac {2}{3}+\frac {\varepsilon_1}{10}},n+n^{\frac 23+\frac {\varepsilon_1}{10}}\r]$ and $I'_{n}=\l\{n^{-\varepsilon_1/100}\r\}\times \l[n^{1-\frac{\varepsilon_1}{100}}-n^{\frac 23+\frac{\varepsilon_1}{10}}, n^{1-\frac{\varepsilon_1}{100}}+n^{\frac 23+\frac{\varepsilon_1}{10}}\r] $. Consider the parallelogram $P_{n}$ with opposite sides $I_{n}$ and $I'_{n}$. Define also the sub-interval $J_{n}$ of $I'_{n}$ given by $J_n=\l\{n^{-\varepsilon_1/100}\r\}\times \l[n^{1-\frac{\varepsilon_1}{100}}-n^{\frac 23-\frac{\varepsilon_1}{300}}, n^{1-\frac{\varepsilon_1}{100}}+n^{\frac 23-\frac{\varepsilon_1}{300}}\r]$ (see Figure \ref{fig:liminf upper bound}). 

For $v=(t,i)\in \R_{+}\times \N$, let us for brevity of notation set $f(v)=2\sqrt{ti}$. Let us set 
$$Y_{n}=\max_{v\in J_{n}} \l(f(v)+D^{P_n}_{v\to (1,n)}\r)$$
where $D^{P_n}_{v\to (1,n)}$ denotes the maximum passage time from $v$ to $(1,n)$ where the maximum is taken over all paths contained in $P_n$. Recall the definition of $u_k$ in Proposition \ref{p:decorr}. We shall show in the next couple of lemmas that if we set $X_{k}=Y_{u_k}$ for all $k$ sufficiently large, it satisfies the conditions in Proposition \ref{p:decorr}; thus completing the proof of the proposition. 
\qed
\begin{figure}[t!]
    \includegraphics[width=15 cm]{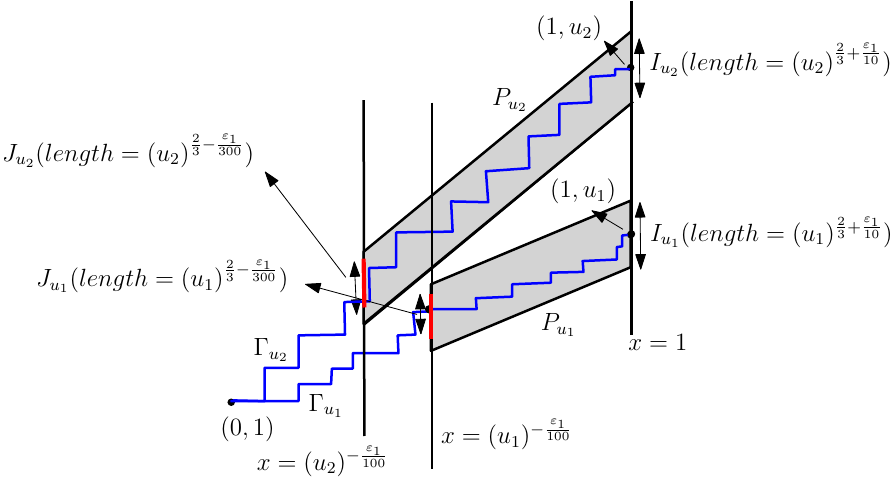}
    \caption{To prove Proposition \ref{p:decorr} we consider the points $(1,u_1),(1,u_2), \dots$ on the line $x=1$ where $u_k=k^{3+\vep_1}$. Since the parallelogram $P_{u_k}$ (shaded in grey) are disjoint Lemma \ref{l:indep} is immediate. Thus to prove the proposition it is sufficient to show Lemma \ref{l:approx}. To do this we fix a $k$ and construct the events $\ca,\cb$ and $\cc.$ On the event $\ca, \Gamma_{1, {u_k}}$ stays within the parallelogram $P_{u_k}$. By Lemma \ref{lemma: transversal fluctuation}, (i) $\ca^c$ has exponentially small (in $n$) probability. $\cb$ is the event that $\Gamma_{1,u_k}$ intersects the line $x=n^{-\frac{\vep_1}{100}}$ inside the interval $J_{u_k}$ (shaded in red). By Lemma \ref{lemma: transversal fluctuation}, (ii) $\cb^c$ has exponentially small (in $n$) probability. Finally, $\cc$ is the event that the maximum (resp.\ minimum) passage times between $(0,1)$ and $J_{u_k}$ are not too large (resp.\ not too small). By Proposition \ref{p: upper tail estimates} and Proposition \ref{p: lower tail estimates} and a union bound over $J_{u_k}$ we conclude that $\cc^c$ also has exponentially small (in $n$) probability. Now on the event $\ca \cap \cb \cap \cc,$ $Y_{u_k}$ and $D_{1,u_{k}}$ are close on the correct scale.}
   \label{fig:liminf upper bound}
 \end{figure}

\begin{lemma}
    \label{l:indep}
    There exists $k_0=k_0(\varepsilon_1)$ such that the random variables $Y_{u_{k}}$ are independent for $k>k_0$. 
\end{lemma}

\begin{proof}
    It is easy to check that for all $k$ sufficiently large the parallelograms $P_{u_k}$ are disjoint. Since $Y_{u_k}$ depends only on the Brownian increments in the parallelogram $P_{u_k}$, the claim follows. 
\end{proof}

\begin{lemma}
    \label{l:approx}
    For ${\vep_1}, \varepsilon_2>0$ small enough, there exists $\delta>0$ such that for all $n$ sufficiently large we have 
    $$\P\l(\l|Y_n-D_{1,n}\r|\ge \varepsilon_2 n^{-1/6}\r) \le \exp\l(-n^{\delta}\r).$$
\end{lemma}

\begin{proof}
    Consider the event $\ca$ that for all $v\in J_n$ the geodesic from $v$ to $(1,n)$ is contained in $P_n$. Next, consider the event 
    $$\cb=\l\{\Gamma_{1,n}\l(n^{-\varepsilon_1/100}\r) \subseteq J_{n}\r\},$$
    i.e., the geodesics $\Gamma_{1,n}$ intersects the line $x=(n^{-\varepsilon_1/100})$ within the interval $J_n$. Finally, consider the event 
    $$\cc=\left \{\max_{v\in J_n}: |D_v-f(v)|\le \varepsilon_2 n^{-1/6} \right \}.$$
    Clearly, on the event $\ca\cap \cb \cap \cc$ we have 
    $|D_{1,n}-Y_n|\le \varepsilon_2 n^{-1/6}$. This is because on $\ca \cap \cb \cap \cc$ there exist $v \in J_n$ such that 
    \[
    D_{1,n}=D_v+D^{P_n}_{v \rightarrow (1,n)}.
    \]
    Hence, 
    \[
    D_{1,n} \leq D_v-f(v)+Y_n \leq \max_{v \in J_n}|D_v-f(v)|+Y_n.
    \]
    Also by superadditivity for all $v \in J_n$
    \[
    D_{1,n} \geq D_v+D^{P_n}_{v\rightarrow (1,n)}.
    \]
    Hence, 
    \[
    D_{1,n} \geq -\max_{v \in J_n}|D_v-f(v)|+Y_n.
    \]
    This gives us the desired upper bound for $|D_{1,n}-Y_n|.$ So it suffices to show 
    $\P(\ca^c\cup \cb^c \cup \cc^c)\le \exp(-n^{\delta})$. We first start with the event $\cc^c.$ 
    \[
    \P \left(\cc^c \right) \leq \P \left(\max_{v \in J_n}\left(D_v-f(v) \right) \geq \vep_2n^{-1/6} \right)+\P \left(\min_{v \in J_n}\left(D_v-f(v) \right) \leq -\vep_2n^{-1/6} \right).
    \]
    For the first probability on the right side note that By Proposition \ref{p: upper tail estimates} and a union bound we get there exists $c>0$ such that for $n$ sufficiently large depending on $\vep_1,\vep_2$
    \[
    \P \left(\max_{v \in J_n}\left(D_v-f(v) \right) \geq \vep_2n^{-1/6} \right) \leq 2 n\exp \left(-c n^{\delta} \right)
    \]
    Similarly, by Proposition \ref{p: lower tail estimates} we get there exists $c>0$ such that for all $n$ sufficiently large depending on $\vep_1,\vep_2$
    \[
    \P \left(\min_{v \in J_n}\left(D_v-f(v) \right) \leq -\vep_2n^{-1/6} \right) \leq 2n\exp\left(-c n^{\delta} \right).
    \]
    Combining these we get upper bound for $\P \left(\cc^c \right).$ For the event $\ca^c, $ note that by Lemma \ref{lemma: transversal fluctuation}, (i) we get for $n$ sufficiently large $\P \left(\ca^c \right) \leq n\exp\l(-cn^{\frac{3}{10}\vep_1}\r)$. Finally, applying Lemma \ref{lemma: transversal fluctuation}, (ii) we get $\P \left( \cb^c\right) \leq \exp \left(-n^{\frac{\vep_1}{100} }\right).$ Combining all the upper bounds the proof is completed.
      \end{proof}

\subsection{Proof of Lemma \ref{lemma: transversal fluctuation}}
We now proceed with the remaining proof of the transversal fluctuation estimates in Lemma \ref{lemma: transversal fluctuation}.

\begin{proof}[Proof of Lemma \ref{lemma: transversal fluctuation}] (i) Note that if the event $\{\mathsf{TF}_n \geq \ell n^{2/3} \}$ happens then then there exists $t_0 < 1$ such that $(t_0,t_0n +xn^{2/3}) \in \Gamma_{1,n}$ where $|x| \geq \ell.$ By Observation \ref{p: scaling}, (ii) 
\begin{align*}
&\P \left(\left \{(t_0, t_0n +xn^{2/3}) \in \Gamma_{1,n} \text{ for some } t_0 <1 \text{ and } |x| \geq \ell \right\}\right)\\&=\P \left(\left \{(t_0n,  t_0n+xn^{2/3}) \in \Gamma_{n,n} \text{ for some } t_0<1 \text{ and }|x| \geq \ell \right\}\right).
\end{align*}
By \cite[Corollary 1.5]{GanHam23} the probability on the right hand side is at most $\exp(-c \ell^3)$ for some $c>0$ and $n, \ell$ sufficiently large with $\ell \leq n^{1/10}$. This completes the proof of the first part.\\
\begin{figure}[t!]
    \includegraphics[width=15 cm]{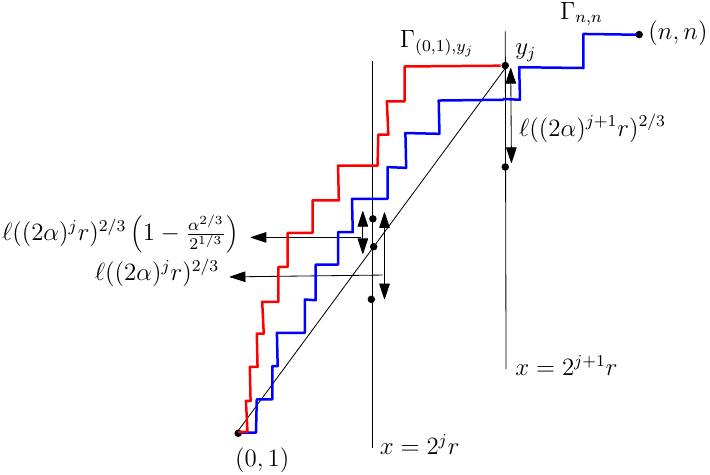}
    \caption{To prove Lemma \ref{lemma: transversal fluctuation}, (ii) we consider a union of events $\mathcal{A}_j$'s on which the distance between $\Gamma_{n,n}(2^jr)$ and $2^jr$ is at least $\ell ((2 \alpha)^jr)^{2/3}$ and the distance between $\Gamma_{n,n}(2^{j+1}r)$ and $2^{j+1}r$ is at most $\ell ((2 \alpha)^{j+1}r)^{2/3}$. Using the first part of the lemma it is enough to consider $\bigcup_{j} \ca_j.$ We fix $j$ and show that on the event $\ca_j,$ the geodesic between $(0,1)$ and $y_j$ has large global transversal fluctuation from the straight line joining $(0,1)$ and $y_j$. Thus, applying first part of the lemma to each of the $\ca_j$'s and applying a union bound gives us the desired upper bound.}
    \label{fig:local transversal fluctuation}
\end{figure}
(ii) Note that by Observation \ref{p: scaling}, (ii) we have 
\[
\P \left(\left \{\left (t, tn+x (tn)^{2/3} \right) \in \Gamma_{1,n} \text{ for some } |x| \geq \ell \right \}\right)=\P \left(\left \{\left (tn, tn+x (tn)^{2/3} \right) \in \Gamma_{n,n} \text{ for some } |x| \geq \ell \right \}\right).
\]
Thus it is enough to prove the following: there exists $c>0$ such that for all $n, \ell,r$ sufficiently large with $r \leq n$, 
\[
\P \left(\left \{\left (r, r+x r^{2/3} \right) \in \Gamma_{n,n} \text{ for some } |x| \geq \ell \right \}\right) \leq \exp(-c \ell ^3).
\] The proof idea will be same as \cite[Proposition 2.1]{BBB23}. We prove the above for $x>0$. The argument for the negative $x$ case follows similarly. Further note that when $r \geq \frac n2,$ then the upper bound follows from the first part. So, we only need to consider the case when $r \leq \frac n2.$
For a fixed $1 < \alpha <\sqrt{2}$, we define the following events. For all $0 \leq j \leq \lceil \log_2 \left(\frac nr \right)\rceil-1,$ let (see Figure \ref{fig:local transversal fluctuation})
\[
\ca_j:=\left \{\Gamma_{n,n}\left(2^jr \right)-2^jr \geq \ell ((2\alpha)^j r)^{2/3} \text{ and } \Gamma_{n,n}(2^{j+1}r)-2^{j+1}r \leq \ell ((2 \alpha)^{j+1}r)^{2/3}\right \}
\]
Let $\ca:=\bigcup_{j}\ca_j.$ Note that 
\[
\ca^c \cap  \left \{ \Gamma_{n,n}(r)-r \geq \ell r^{2/3}\right\} \subset \left \{ \Gamma_{n,n}\left(2^jr \right)-2^jr \geq \ell ((2\alpha)^j r)^{2/3} \text{ for all } j\right \}
\]
Let $j_0=\lceil \log_2 \left(\frac nr \right)\rceil-1.$ Then by the first part, and by the choice of $j_0$ the event\\ $\left \{ \Gamma_{n,n}\left(2^{j_0}r \right)-2^{j_0}r \geq \ell ((2\alpha)^{j_0} r)^{2/3}\right \}$ has probability at most $\exp(-c \ell^3)$ for all sufficiently large $n.$ Thus, it is enough to find an upper bound for $\P\left(\ca \right).$ We fix a $j.$ Consider the point $y_j=(2^{j+1}r,2^{j+1}r+\ell \left((2\alpha)^{j+1}r\right)^{2/3})$. Then we see that by Observation \ref{o: ordering} if the event $\ca_j$ happens then it must happen that  $\Gamma_{(0,1),y_j}(2^jr)-2^jr \geq \ell \left((2 \alpha)^jr\right)^{2/3}$. Otherwise let $v=(2^{j+1}r, m)$ be the random point in $\Gamma_{n,n}$. Then there will be some $m_0 \leq m$ and $0 \leq x, y \leq 2^{j+1}r $ such that $(x,m_0) \in \Gamma_{(0,1),y_j}$ and $(y,m_0) \in \Gamma_{(0,1),v}$ but $x>y$. This is a contradiction to Observation \ref{o: ordering} (see Figure \ref{fig:local transversal fluctuation}). Hence, 
\begin{align*}
&\ca_j \subset \left \{\Gamma_{(0,1),y_j}(2^jr)-2^jr \geq \ell \left((2 \alpha)^jr  \right)^{2/3} \right \}\\
&= \left \{\Gamma_{(0,1),y_j}(2^j r)-\left(2^jr+\frac{\ell ((2\alpha)^{j+1}r)^{2/3}}{2} \right) \geq \ell ((2 \alpha)^jr)^{2/3}\left(1-\frac{\alpha^{2/3}}{2^{1/3}} \right)  \right\}.
\end{align*}
As $\alpha < \sqrt{2} $ then we are in the set-up of the first part. Now, if we choose $r \geq \max\{n_0, \ell_0^{1/10}\}, \ell \geq \ell_0$ and $\ell \leq r^{1/10}$ then for all $j$ the conditions in the first part are satisfied. Hence, we get there is $c>0$ such that for all $j,$ and $\ell \leq r^{1/10}$ sufficiently large
\[
\P (\ca_j) \leq \exp(-c\ell^3 \alpha^{2j}).
\]
Hence, 
\[
\P \left(\ca \right) \leq \sum_{j} \exp \left(-c\ell^3 \alpha^{2j} \right).
\]
Finally, as $\alpha>1$ the right hand side is smaller than $\exp(-c\ell^3).$ This completes the proof.\end{proof}

\section{Law of fractional logarithm in Poissonian LPP}
\label{s:lisfull}

In this section we prove Proposition \ref{l:LFL for sequence}. Before going into the proof we first state some auxiliary estimates for Poissonian LPP.
\subsection{Auxiliary Estimates:}
We first begin with the sharp tail estimates for $\cl_{1,t}.$
\begin{theorem}{\cite[Theorem 1.3]{LM01}}\label{thm: mod dev upper tail}
For any fixed $\delta,\vep>0$, there exist $t_{\vep,\delta}$ and $x_{\vep,\delta}$ such that for all $t\geq t_{\vep,\delta}$ and $x_{\vep,\delta}\leq x\leq t^{\frac{1}{3}-\delta}$, we have
\begin{align}\label{eq: mod dev upper tail}
    \exp\l(-\frac{4}{3}(1+\vep)x^{3/2}\r)\leq \P\l(\l(\mathcal{L}_{1,t}-2\sqrt{t}\r)t^{-1/6}\geq x\r)\leq \exp\l(-\frac{4}{3}(1-\vep)x^{3/2}\r).
\end{align}
\end{theorem}

 \begin{theorem}{\cite[Theorem 1.2]{LMS02}}\label{thm: mod dev lower tail}
 For any fixed $\delta,\vep>0$, there exist $t_{\vep,\delta}$ and $x_{\vep,\delta}$ such that for all $t\geq t_{\vep,\delta}$ and $x_{\vep,\delta}\leq x\leq t^{\frac{1}{3}-\delta}$, we have
 \begin{align}\label{eq: mod dev lower tail}
     \exp\l(-\frac{1}{12}(1+\vep)x^{3}\r)\leq \P\l(\l(\mathcal{L}_{1,t}-2\sqrt{t}\r)\lambda^{-1/6}\leq -x\r)\leq \exp\l(-\frac{1}{12}(1-\vep)x^{3}\r).
 \end{align}
\end{theorem}
Next we state the scaling of Poissonian LPP.
\begin{observation}{(Poissonian Scaling)}
\label{o:poissonian scaling}
For all $s>0,$
\[
\left \{\cl_{1,t}\right \}_{t \geq 0} \overset{d}{=}\left \{\cl_{s,\frac ts} \right \}_{t \geq 0}.
\]
\end{observation}
Now we define the notion of geodesics in Poissonian LPP. For any fixed pair of coordinate wise ordered points $a=(a_1,a_2)<b=(b_1,b_2)$, we consider piecewise linear paths joining coordinate wise ordered points $u_0=a, u_1,u_2,\ldots u_{k-1},u_k=b$ where each $u_1,u_2,\ldots, u_{k-1}$ is a point in the Poisson process $\Pi$. Since for a Poisson process on $\R^2$, on any vertical line there will not be more than one Poisson point, any increasing path $\gamma$ between $a=(a_1,a_2)<b=(b_1,b_2)$ can be considered as an increasing function on $[a_1,b_1]$ (i.e., $(x,\gamma(x))$ denotes the unique point on $\gamma$ with first coordinate $x$). Between any pair of ordered points in $\R^2$, increasing paths with maximal number of Poisson points exist, although they may not be unique. Such paths will be called geodesics. For $a,b \in \R^2 $ with $a=(a_1,a_2)<b=(b_1,b_2), \Gamma_{a,b}:=\{(x,\Gamma_{a,b}(x)): x \in [a_1,b_1] \}$ will be used to denote some maximal increasing path from $a$ to $b$. We state the following estimate for transversal fluctuations of geodesics. For any $r>0$ denote 
\[
D(r):=\sup\{|\Gamma_{r,r}(x)-x|: x \in [0,r], \Gamma_{r,r} \text{ is a geodesic between } (0,0) \text{ to } (r,r)\}.
\] 
We have the analogue of Lemma \ref{lemma: transversal fluctuation}.
\begin{proposition}
\label{t:transversal fluctuation in diagonal direction}We have the following: 
    \begin{enumerate}[label=(\roman*), font=\normalfont]
    \item  There exist $r_0,\ell_0$ and $c>0$ such that for all $r \geq r_0, \ell \geq \ell_0$ we have\[
    \P(D(r)\geq \ell r^{2/3}) \leq e^{-c\ell}.
    \]
    \item Let $r_0,\ell_0$ be as above. Then there exists $c'>0$ such that for all $r \geq r_0, \ell\geq \ell_0$ and for all $t \in [0,r]$ with $t \geq r_0$
    \[
    \P(|\Gamma_{r,r}(t)-t| \geq \ell t^{2/3} \text{ for some geodesic } \Gamma_{r,r} \text{ between } (0,0) \text{ and } (r,r)) \leq e^{-c'\ell}.
    \]
    \end{enumerate}
    \end{proposition}
\begin{proof} (i) The statement follows from \cite[Theorem 11.1]{BSS14}, where the estimate was proved for topmost geodesic. By symmetry one can conclude the same for the bottom most geodesic. Combining these we have the desired estimate.

(ii) Using (i) and ordering of geodesics \cite[Lemma 11.2]{BSS14}, the proof of (ii) is similar to the proof of Lemma \ref{lemma: transversal fluctuation}, (ii). To avoid repetition we omit the details (see also \cite{BSS17B, BBB23} for a similar proof for the exponential LPP). 
    \end{proof}

\begin{remark}
    Using similar arguments as in the Brownian LPP or exponential LPP case, case it is expected to have an upper bound of $e^{-c \ell^3}$ in both the statements above. However, we are quoting the results from \cite{BSS14} where the weaker version (i) as stated above was proved. The bounds as stated will be sufficient for our purpose.
\end{remark}

Finally we state the transversal fluctuation estimates for the geodesics between $(0,0)$ and $(1,r).$ For any $r>0$ denote $$D_1(r):=\sup\{|\Gamma_{1,r}(x)-rx|: x \in [0,1], \Gamma_{1,r} \text{ is a geodesic between } (0,0) \text{ to } (1,r)\}.$$ We have the following.

\begin{lemma}
    \label{lemma: Poissonian Transversal fluctuation}We have the following: 
    \begin{enumerate}[label=(\roman*), font=\normalfont]
    \item  There exist $r_1,\ell_1$ and $c$ such that for all $r \geq r_1, \ell \geq \ell_1$ we have\[
    \P(D_1(r)\geq \ell r^{5/6}) \leq e^{-c\ell}.
    \]
    \item Let $r_1,\ell_1$ be as above. Then there exists $c'>0$ such that for all $r \geq r_1, \ell\geq \ell_0$ and for all $t \in [0,1]$ with $\sqrt{r}t \geq r_1$
    \[
    \P(\{|\Gamma_{1,r}(t)-rt| \geq \ell r^{5/6}t^{2/3} \text{ for some geodesic } \Gamma_{1,r}\text{ between } (0,0) \text{ and } (1,r)\}) \leq e^{-c'\ell}.
    \]
    \end{enumerate}
\end{lemma}
\begin{proof}
(i) By Observation \ref{o:poissonian scaling} $D_1(r)\overset{d}{=}\sqrt{r}D(\sqrt{r}) $. Therefore,
\[
\P(D_1(r) \geq \ell r^{5/6})=\P\l(\sqrt{r}D(\sqrt{r}) \geq \ell r^{5/6}\r)=\P\l(D(\sqrt{r}) \geq \ell r^{1/3}\r) \leq e^{-c \ell}
\]
for sufficiently large $r$ and $\ell.$ The last inequality follows from Proposition \ref{t:transversal fluctuation in diagonal direction}, (i).\\
(ii) Again by Observation \ref{o:poissonian scaling}, for $t$ as in the statement,
\begin{align*}
    &\P(\{|\Gamma_{1,r}(t)-rt| \geq \ell r^{5/6}t^{2/3} \text{ for some geodesic} \Gamma_{1,r}\text{ between } (0,0) \text{ and } (1,r)\})\\
    &=\P(\{|\sqrt{r}\Gamma_{\sqrt{r},\sqrt{r}}(\sqrt r t)-rt| \geq \ell r^{5/6}t^{2/3} \text{ for some geodesic } \Gamma_{\sqrt{r},\sqrt{r}} \text{ between } (0,0) \text{ and } (\sqrt{r},\sqrt{r})\})\\
    & \P(\{|\Gamma_{\sqrt{r},\sqrt{r}}(\sqrt rt)-\sqrt{r}t| \geq \ell (\sqrt{r}t)^{2/3} \text{ for some geodesic }\Gamma_{\sqrt r, \sqrt r} \text{ between } (0,0) \text{ and } (\sqrt{r},\sqrt{r})\})\\
    &\leq e^{-c'\ell}
\end{align*}
for sufficiently large $\ell,r$ and $\sqrt{r}t$. The last inequality follows from Proposition \ref{t:transversal fluctuation in diagonal direction}, (ii). This completes the proof. 
\end{proof}
 
\subsection{Proof of Proposition \ref{l:LFL for sequence}}

 To prove Proposition \ref{l:LFL for sequence} we need similar estimates as in Proposition \ref{lem: liminf lower bound}, Proposition \ref{lemma: maximum of point to interval estimates} and Proposition \ref{p:decorr}. We state them now. As mentioned before, the difference with Brownian LPP case is that in this case the decorrelation happens at scale $n^{5/6}$. One reason to expect this is that as we saw in Lemma \ref{lemma: Poissonian Transversal fluctuation} the geodesics fluctuates at a scale $n^{5/6}$. The next three propositions show this precisely. First we state the Propositions analogous to \ref{lem: liminf lower bound} and Proposition \ref{lemma: maximum of point to interval estimates} which shows strong correlation at the scale $k\ll n^{5/6}$.

\begin{proposition}\label{lem: perm liminf lower bound} 
Fix any $\vep,\vep'>0$ small enough there exists $n_{\vep,\vep'}$ and $x_\vep$ such that for all $n\geq n_{\vep,\vep'}$, and $x_{\vep} \leq x \leq n^{1/4}$ we have  
\begin{align*}
    \P\l(\min\limits_{n\leq i\leq n+n^{\frac{5}{6}-\vep'}}\l(\mathcal{L}_{1,i}-2\sqrt{i}\r)i^{-1/6}\leq -x\r)\leq \exp \left(-\frac{1}{12}(1-\vep)x^3 \right).
\end{align*}
\end{proposition}

\begin{proof} By Observation \ref{o:poissonian scaling} it suffices to prove that for any $\vep,\vep'>0$ small enough there exists $n_{\vep,\vep'}$ and $x_\vep$ such that for all $n\geq n_{\vep,\vep'}$, and $x_{\vep} \leq x \leq n^{1/4}$ we have
\begin{equation}
\label{eq: Poisson scaled liminf lower bound}
    \P\l(\min\limits_{n\leq i\leq n+n^{\frac{5}{6}-\vep'}}\l(\mathcal{L}_{\sqrt{n},\frac{i}{\sqrt{n}}}-2\sqrt{i}\r)i^{-1/6}\leq -x\r)\leq \exp \left(-\frac{1}{12}(1-\vep)x^3\right).
\end{equation}
Applying the tail estimates for Poissonian last passage time (Theorem \ref{thm: mod dev lower tail}), the rest of the proof is exactly similar to Proposition \ref{lem: liminf lower bound}. We omit the details here.
\end{proof}

Next we state the result analogous to Proposition \ref{lemma: maximum of point to interval estimates}.
\begin{proposition}
    \label{p: Poissonian limsup upper bound}

    Fix any $\vep,\vep'>0$ small enough there exists $n_{\vep,\vep'}$ and $x_\vep$ such that for all $n\geq n_{\vep,\vep'}$, and $x_{\vep} \leq x \leq n^{1/4}$ we have  
\begin{align*}
    \P\l(\max\limits_{n\leq i\leq n+n^{\frac{5}{6}-\vep'}}\l(\mathcal{L}_{1,i}-2\sqrt{i}\r)i^{-1/6}\geq x\r)\leq \exp \left(-\frac{4}{3}(1-\vep)x^{3/2} \right).
\end{align*}
\end{proposition}
\begin{proof}
By Observation \ref{o:poissonian scaling} it is enough to prove that for any $\vep, \vep'>0$ there exist $n_{\vep,\vep'}, x_{\vep}$ such that for all $n \geq n_{\vep, \vep'}$ and $x_{\vep} \leq x \leq n^{1/20}$
\[
\P \left(\max_{n \leq j \leq n+n^{\frac{5}{6}-\vep'}}\left(\mathcal{L}_{\sqrt{n},\frac{j}{\sqrt{n}}}-2 \sqrt{j}\right)j^{-1/6} \geq x\right) \leq \exp\left(-\frac{4}{3}(1-\vep) x^{3/2}\right).
\]
Applying the tail estimates for Poissonian last passage time (Theorem \ref{thm: mod dev upper tail}) rest of the proof is exactly similar to Proposition \ref{lemma: maximum of point to interval estimates}. We omit the details here.
\end{proof}

Finally analogous to Proposition \ref{p:decorr} we have the following.
\begin{proposition}
    \label{p: Poisson decorr} For all $\vep_1,\vep_2>0$ small enough and $u_k=\lfloor k^{6+\vep_1} \rfloor,$ there exists $\delta=\delta(\vep_1,\vep_2)>0$ and $k_0=k_0(\vep_1)>0$ and a sequence of independent random variables $X_k$ such that for all $k>k_0$
    \[
    \P\left(|X_k-\cl_{1,u_k}| \geq \vep_2 (u_k)^{1/6} \right) \leq \exp(-k^\delta).
    \]
\end{proposition}
\begin{proof} To prove Proposition \ref{p: Poisson decorr} we do the following construction which is similar to the proof of Proposition \ref{p:decorr}. For $n\in \N$, consider the vertical intervals $I_{n}=\{1\}\times \l[n-n^{\frac {5}{6}+\frac {\varepsilon_1}{10}},n+n^{\frac 56+\frac {\varepsilon_1}{10}}\r]$ and $I'_{n}=\l\{n^{-\varepsilon_1/100}\r\}\times \l[n^{1-\frac{\varepsilon_1}{100}}-n^{\frac 56+\frac{\varepsilon_1}{10}}, n^{1-\frac{\varepsilon_1}{100}}+n^{\frac 56+\frac{\varepsilon_1}{10}}\r] $. Consider the parallelogram $P_{n}$ with opposite sides $I_{n}$ and $I'_{n}$. Define also the sub-interval $J_{n}$ of $I'_{n}$ given by \\$J_n=\l\{n^{-\varepsilon_1/100}\r\}\times \l[n^{1-\frac{\varepsilon_1}{100}}-n^{\frac 56-\frac{\varepsilon_1}{300}}, n^{1-\frac{\varepsilon_1}{100}}+n^{\frac 56-\frac{\varepsilon_1}{300}}\r]$ (see Figure \ref{fig:liminf upper bound} for the corresponding figure in case of Brownian LPP). For $v=(t,r)\in \R_{+}\times \R_+$, let us for brevity of notation set $f(v)=2\sqrt{tr}$. Let us set 
$$Y_{n}=\max_{v\in J_{n}} \l(f(v)+\cl^{P_n}_{v\to (1,n)}\r)$$
where $\cl^{P_n}_{v\to (1,n)}$ denotes the maximum number of points on an increasing path from $v$ to $(1,n)$ where the maximum is taken over all points contained in $P_n$. Recall the definition of $u_k$. We shall show in Lemma \ref{l: Poisson indep} and Lemma \ref{l: Poisson approx} that if we set $X_{k}=Y_{u_k}$ for all $k$ sufficiently large, it satisfies the conditions in Proposition \ref{p: Poisson decorr}; thus completing the proof of the proposition.
\end{proof}
\begin{lemma}
    \label{l: Poisson indep}
    There exists $k_0=k_0(\varepsilon_1)$ such that the random variables $Y_{u_{k}}$ are independent for $k>k_0$. 
\end{lemma}

\begin{proof}
    It is easy to check that for all $k$ sufficiently large the parallelograms $P_{u_k}$ are disjoint. Since $Y_{u_k}$ depends only on the Poisson points in the parallelogram $P_{u_k}$, the claim follows.
    \end{proof}
\begin{lemma}
    \label{l: Poisson approx}
    For ${\vep_1}, \varepsilon_2>0$ small enough, there exists $\delta>0$ such that for all $n$ sufficiently large we have 
    $$\P\l(\l|Y_n-\cl_{1,n}\r|\ge \varepsilon_2 n^{1/6}\r) \le \exp\l(-n^{\delta}\r).$$
\end{lemma}

\begin{proof}Consider the event $\ca$ that for all $v\in J_n$, all geodesics $\Gamma_{v, (1,n)}$ are contained in $P_n$. Next, consider the event 
    $$\cb=\l\{\Gamma_{1,n}\l(n^{-\varepsilon_1/100}\r) \subseteq J_{n} \text{ for all geodesics} \Gamma_{1,n} \text{ from } (0,0) \text{ to } (1,n)\r\}$$
    i.e., all geodesics from $(0,0)$ to $(1,n)$  intersect the line $x=(n^{-\varepsilon_1/100})$ within the interval $J_n$. Finally, consider the event 
    $$\cc=\left \{\max_{v\in J_n}: |\cl_v-f(v)|\le \varepsilon_2n^{1/6} \right \}.$$
    Arguing similarly as in the proof of lemma \ref{l:approx}, on the event $\ca\cap \cb \cap \cc$ we have 
    $|\cl_{1,n}-Y_n|\le \varepsilon_2n^{1/6}$. 
   
    So it suffices to show 
    $\P(\ca^c\cup \cb^c \cup \cc^c)\le \exp(-n^{\delta})$. For the probability of the events $\ca^c$ (resp.\ $\cb^c$) we use Lemma \ref{lemma: Poissonian Transversal fluctuation}, $(i)$ (resp.\ $(ii)$). For the probability of the event $\cc^c$ we use \cite[Lemma 10.2, Lemma 10.6]{BSS14}. So, as we have argued in the proof of Lemma \ref{l:approx} we get the desired upper bound in this case.
    \end{proof}
    \begin{proof}[Proof of Proposition \ref{l:LFL for sequence}]
    Using Proposition \ref{lem: perm liminf lower bound}, \ref{p: Poissonian limsup upper bound} and \ref{p: Poisson decorr} we obtain Proposition \ref{l:LFL for sequence} exactly similarly as we have proved Theorem \ref{t:LPP} and Theorem \ref{t:LPPupper}. In particular, for the limsup upper bound using Proposition \ref{p: Poissonian limsup upper bound} we show that for any $\vep>0$ there exists $\delta>0$ such that 
    \[
    \P \left(\max_{n_k \leq i \leq n_{k+1}}  \frac{\left({\cl_{1,i}-2\sqrt{i}}\right){i^{-1/6}}}{(\log i)^{1/3}} \geq \l(\frac 18+\vep\r)^{2/3}\right) \leq \frac{1}{k^{1+\delta}}
    \]
    where $n_k=k^{6-36\vep'}$ for some small enough $\vep'$ depending on $\vep$ (note that this subsequence is different than the one considered in the proof of Theorem \ref{t:LPPupper}. This is because the decorrelation happens in this case at scale $n^{5/6}$). This proves the limsup upper bound. For the limsup lower bound we show that 
    $$\limsup_{k\to \infty} \dfrac{u_k^{-1/6}\l(Y_{k}-2\sqrt{u_k}\r)}{(\log u_k)^{1/3}}\ge \l(\frac 18 \r)^{2/3},$$
    almost surely, where $Y_k$ is as in Proposition \ref{p: Poisson decorr} and $u_k$ is now $\lfloor k^{6+\vep_1}\rfloor$ as defined in Proposition \ref{p: Poisson decorr}. Arguing similarly as before we obtain the lower bound. The proof of liminf follows similarly by applying Proposition \ref{lem: perm liminf lower bound} and \ref{p: Poisson decorr}. 
        \end{proof}
       
\section{The 0-1 laws}
\label{s:01}

We end with the proof of Proposition \ref{p:01}.
\begin{proof}[Proof of Proposition \ref{p:01}]
We prove the first part of the Proposition for limsup. The result for the liminf follows similarly. For $n\geq i$, we let $T_{n}^{(i)}=D_{(0,i)\rightarrow (1,n)}$ denote the passage time to $(1,n)$ using only the Brownian motions $\{B_{\ell}\}_{\ell\geq i}$. In particular $T_{n}^{(1)}$ equals $D_{1,n}$. It is immediate that almost surely $T_{n}^{(i)}$ is non-increasing in $i$ for all $i\leq n$. Hence we have that almost surely the random variables
\begin{align*}
\xi_i:=\limsup\limits_{n\rightarrow\infty}\frac{\left({T_{n}^{(i)}-2\sqrt{n}}\right){n^{1/6}}}{f(n)}
\end{align*}
are non-increasing in $i$. Also note that for fixed $i, (\sqrt{n}-\sqrt{n-i})n^{1/6}=O(n^{-1/3})$ as $n \rightarrow \infty$. Hence, by the conditions on $f$ we also have 
\begin{align*}
\xi_i=\limsup\limits_{n\rightarrow\infty}\frac{\left({T_{n}^{(i)}-2\sqrt{n-i}}\right){(n-i)^{1/6}}}{f(n-i)}.
\end{align*}
Therefore, $(\xi_i)_{i\geq 1}$ is a stationary  sequence.   As $\xi_{i+1}\le \xi_{i}$ 

and the two have the same distribution, $\xi_i=\xi_1$ for all $i$, almost surely. As $\xi_i$ depends on $\{B_\ell\}_{\ell \geq i}$ it follows that $\xi_1 \in \bigcap_{i \geq 1}\sigma \left( B_{i},B_{i+1}, \dots  \right)$. Hence, $\xi_1$ is tail measurable for the i.i.d. sequence $(B_{\ell})_{\ell \geq 1}$. It follows that that $\xi_1$ is a constant random variable.

The proof for Poissonian case is similar. Consider the set-up of Poissonian LPP. For $n\geq i$, we let $\cl_{1,t}^{(i)}=\cl_{(0,i)\rightarrow (1,t)}$ denote the passage time to $(1,t)$ using only the Poisson points above the line $y=i$. In particular $\cl_{1,t}^{(0)}$ equals $\cl_{1,t}$. It is immediate that almost surely $\cl_{1,t}^{(i)}$ is non-increasing in $i$ for all $i\leq t$. Hence we have that almost surely the random variables
\begin{align*}
\psi_i:=\limsup\limits_{t\rightarrow\infty}\frac{\left({\cl_{1,t}^{(i)}-2\sqrt{t}}\right)}{t^{1/6}f(t)}
\end{align*}
are non-decreasing. Arguing as before we also get that $(\psi_i)_{i \geq 0}$ is a stationary sequence. Hence, in this case also $\psi_i=\psi_0$ for all $i$ almost surely. As $\psi_i$ depends on Poisson points from above the line $y=i,$ it follows that $\psi_0$ is tail measurable for the i.i.d. sequence $\zeta_i=\Pi\cap \l( [0,1]\times[i,i+1)\r) $ where $\Pi$ is the underlying Poisson point process. This shows that $\psi_0$ is constant almost surely. The proof of the liminf case is similar.
\end{proof}

 \bibliography{references}
 \bibliographystyle{plain}

\end{document}